\documentclass[a4paper,11pt,oneside]{amsart}
\pdfoutput=1

\usepackage[utf8]{inputenc}
\usepackage{bm}
\usepackage{mathtools,amssymb}
\usepackage{esint}
\usepackage{hyperref}
\hypersetup{
    colorlinks=true,
    linkcolor=black,
    filecolor=black,      
    urlcolor=cyan,
    citecolor=black
    }
\usepackage{tikz}
\usetikzlibrary{arrows.meta}
\usepackage{dsfont}
\usepackage{relsize}
\usepackage{url}
\usepackage{xcolor}
\usepackage{graphicx}
\usepackage{mathrsfs}
\usepackage[shortlabels]{enumitem}
\usepackage{lineno}
\usepackage{amsmath}
\usepackage{enumitem}
\usepackage{amsthm} 
\usepackage{verbatim}
\usepackage{dsfont}

\allowdisplaybreaks

\mathtoolsset{showonlyrefs}

\graphicspath{{images/}}

\newtheorem{theorem}{Theorem}[section]
\newtheorem{lemma}[theorem]{Lemma}
\newtheorem{definition}[theorem]{Definition}
\newtheorem{corollary}[theorem]{Corollary}
\newtheorem{conjecture}[theorem]{Conjecture}

\newtheorem{proposition}[theorem]{Proposition}
\newtheorem{question}[theorem]{Question}
\newtheorem{remark}[theorem]{Remark}

\title[The inverse fractional conductivity problem]{Low regularity theory for the inverse fractional conductivity problem}

\author{Jesse Railo}
\thanks{Department of Pure Mathematics and Mathematical Statistics, University of
Cambridge (\href{mailto:jr891@cam.ac.uk}{jr891@cam.ac.uk})}
\address{Department of Pure Mathematics and Mathematical Statistics, University of
Cambridge, Cambridge CB3 0WB, UK}
\email{jr891@cam.ac.uk}
\author{Philipp Zimmermann}
\thanks{Department of Mathematics, ETH Zurich (\href{mailto:philipp.zimmermann@math.ethz.ch}{philipp.zimmermann@math.ethz.ch})}
\address{Department of Mathematics, ETH Zurich, Z\"urich, Switzerland}
\email{philipp.zimmermann@math.ethz.ch}
\date{\today}

\newcommand{\R}{{\mathbb R}}

\newcommand{\N}{{\mathbb N}}

\newcommand{\schwartz}{\mathscr{S}}

\newcommand{\tempered}{\mathscr{S}^{\prime}}

\newcommand{\fourier}{\mathcal{F}}
\newcommand{\ifourier}{\mathcal{F}^{-1}}
\newcommand{\vev}[1]{\left\langle#1\right\rangle}

\newcommand{\cdistr}{\mathscr{E}'}
\newcommand{\distr}{\mathscr{D}^{\prime}}
\newcommand{\norm}[1]{\lVert #1 \rVert}
\newcommand{\abs}[1]{\left\lvert #1 \right\rvert}
\newcommand{\ip}[2]{\left\langle #1,#2 \right\rangle}
\DeclareMathOperator{\Div}{div} 
\DeclareMathOperator{\supp}{supp} 
\DeclareMathOperator{\dist}{dist} 

\begin{document}

\maketitle
\begin{abstract} We characterize partial data uniqueness for the inverse fractional conductivity problem with $H^{s,n/s}$ regularity assumptions in all dimensions. This extends the earlier results for $H^{2s,\frac{n}{2s}}\cap H^s$ conductivities by Covi and the authors. We construct counterexamples to uniqueness on domains bounded in one direction whenever measurements are performed in disjoint open sets having positive distance to the domain. In particular, we provide counterexamples in the special cases $s \in (n/4,1)$, $n=2,3$, missing in the literature due to the earlier regularity conditions. We also give a new proof of the uniqueness result which is not based on the Runge approximation property. Our work can be seen as a fractional counterpart of Haberman's uniqueness theorem for the classical Calderón problem with $W^{1,n}$ conductivities when $n=3,4$. One motivation of this work is Brown's conjecture that uniqueness for the classical Calderón problem holds for $W^{1,n}$ conductivities also in dimensions $n \geq 5$.
\end{abstract}


\section{Introduction}
We study the global inverse fractional conductivity problem in a low regularity setting which extends the earlier theory in \cite{C21,RGZ2022GlobalUniqueness,RZ2022unboundedFracCald}. The considered setting resembles the classical inverse conductivity problem with $W^{1,n}$ conductivities \cite{H15}. The inverse conductivity problem is also known as the \emph{Calderón problem} due to the seminal work of Calderón \cite{C80}, first published in 1980. The classical Calderón problem forms the mathematical model of electrical impedance tomography \cite{UH-inverse-problems-seeing-the-unseen}. The mathematical studies of the inverse conductivity problem date at least back to the work of Langer \cite{Langer33-calderon-halfspace}. The proof of uniqueness for the classical Calderón problem is based on the reduction to the inverse problem for the Schrödinger equation (called \emph{Liouville reduction}), on the construction of complex geometric optics (CGO) solutions \cite{SU87}, and on boundary determination results \cite{KV84}. The proof of global uniqueness for the fractional Calderón problem is based on similar ideas \cite{C21,RGZ2022GlobalUniqueness,RZ2022unboundedFracCald} but the construction of CGO solutions can be replaced by the UCP of the fractional Laplacians \cite{GSU20,RS-fractional-calderon-low-regularity-stability} and boundary determination results are replaced by exterior determination results, using sequences of special solutions whose energies can be concentrated in the limit to any point in the exterior \cite{RGZ2022GlobalUniqueness,RZ2022unboundedFracCald}.

Let $V \subset \R^n$ be a measurable set. We say that $\gamma \in L^\infty(V)$ is \emph{uniformly elliptic} if $a\leq \gamma \leq b$ a.e. in $V$ for some $0<a<b<\infty$.
We assume that all conductivities are uniformly elliptic throughout the article. Suppose that the electric potential $u_f\in H^1(\Omega)$ solves the Dirichlet problem for the classical conductivity equation
\begin{equation}
\label{local-conductivity}
\begin{split}
    \Div(\gamma\nabla u)&=0 \quad \mbox{ in } \Omega, \\ u&=f \quad \mbox{ on } \partial\Omega
\end{split}
\end{equation}
for a given boundary voltage $f\in H^{1/2}(\partial\Omega)$. In the Calderón problem, one aims to recover the conductivity $\gamma$ from the knowledge of the voltage/current measurements given in the form of the Dirichlet-to-Neumann (DN) map $\Lambda_\gamma: H^{1/2}(\partial\Omega)\rightarrow H^{-1/2}(\partial\Omega)$ mapping $f \mapsto \gamma\partial_\nu u_f|_{\partial\Omega}$. 

Let now $s\in(0,1)$ and consider the Dirichlet problem for the fractional conductivity equation
\begin{equation}\begin{split}\label{nonlocal-conductivity}
    \mbox{div}_s (\Theta_\gamma\nabla^su)&=0 \quad \mbox{ in } \Omega, \\ u&=f \quad \mbox{ in } \Omega_e ,
\end{split}\end{equation}
where $\Omega_e\vcentcolon=\R^n\setminus\overline{\Omega}$ is the exterior of the domain $\Omega$, $\Theta_\gamma$ is an appropriate matrix depending on the conductivity $\gamma$, and  $\nabla^s$, $\Div_s$ are the fractional gradient, divergence, respectively (see~\cite{RGZ2022GlobalUniqueness, C21, RZ2022unboundedFracCald}). The inverse problem for the fractional conductivity equation asks to recover the conductivity $\gamma$ from a nonlocal analogue of the DN data, which in the case of Lipschitz domains maps as $\Lambda^s_\gamma \colon H^s(\Omega_e)\rightarrow H^{-s}_{\overline \Omega_e}(\R^n)$ (see Lemma \ref{lemma: well-posedness results and DN maps} for the general definition). We say $u\in H^s(\R^n)$ is a (weak) solution of \eqref{nonlocal-conductivity} if there holds
\begin{equation}\label{eq:generalNonlocalOperators}
    B_\gamma(u,\phi):=\frac{C_{n,s}}{2}\int_{\R^{2n}} \frac{\gamma^{1/2}(x)\gamma^{1/2}(y)}{\abs{x-y}^{n+2s}} (u(x)-u(y))(\phi(x)-\phi(y))\,dxdy =0
\end{equation} 
for all $\phi \in C_c^\infty(\Omega)$ and $u-f \in \tilde{H}^s(\Omega)$. See for instance the survey~\cite{RosOton16-NonlocEllipticSurvey} on more general nonlocal equations of similar type, and the nonlocal vector calculus developed in \cite{DGLZ12}. As $s \uparrow 1$, then the fractional conductivity operator converges in the sense of distributions to the classical conductivity operator when applied to sufficiently regular functions (cf.~\cite[Lemma 4.2]{C21}). In the rest of the article, we will simply write $\Lambda_\gamma$ instead of $\Lambda^s_\gamma$ also in the fractional case. We also define $m_\gamma \vcentcolon = \gamma^{1/2}-1$ and call it the \emph{background deviation} of $\gamma$.

The fractional Calderón problems have been studied intensively in the past few years~\cite{S17}, starting from the work of Ghosh, Salo, and Uhlmann \cite{GSU20}. The research in this area has mainly focused on the recovery of additive perturbations of a priori known nonlocal operators from the exterior DN maps \cite{C21,CMRU22,LL-fractional-semilinear-problems,RZ2022unboundedFracCald}. More recently, there has been growing interest towards inverse problems for nonlocal variable coefficient operators \cite{covi2019inverse-frac-cond,RGZ2022GlobalUniqueness,ghosh2021calderon} and for nonlocal geometric problems \cite{feizmohammadiEtAl2021fractional}. Inverse problems for time-fractional, space-fractional and spacetime-fractional equations are considered recently for example in \cite{KowLinYiHsuanWang2022FracWave,LaiLinRuland2020FracPar,HelinLassasYlinenZhangSpaceTimeFracHeat,Li2021FracParMagn,KianLiuYamamoto2021GenEvol,BanerjeeKrishnanSenapati2022VarCoeffSpaceTimePar}. More references of nonlocal inverse problems can be found from the aforesaid works.

\subsection{Regularity theory for the inverse conductivity problem and Brown's conjecture} In this subsection, we briefly recall some of the important contributions to low regularity theory for the inverse conductivity problem. More references can be found from the cited works. 

One important step in the uniqueness of classical Calderón problem is the determination of conductivities on the boundary of a domain. Kohn and Vogelius proved that the boundary jet of a smooth function can be uniquely determined from the local DN map \cite{KV84}. They applied this result to show uniqueness for piecewise analytic conductivities when $n \geq 2$ \cite{KV85}. Sylvester and Uhlmann proved uniqueness for $C^2$ conductivities when $n \geq 3$ in their celebrated work \cite{SU87}. This result was improved to cover $C^1$ conductivities by Haberman and Tataru~\cite{HT13} and extended to Lipschitz conductivities by Caro and Rogers~\cite{CR16}. 

A low regularity boundary determination result for conductivities, which are continuous a.e. on $\partial \Omega$, was obtained by Brown (see~\cite{Brown} for the precise formulation). The results in~\cite{Brown} improved the earlier boundary determination results for $C^\infty$ conductivities \cite{KV84}, Lipschitz by Alessandrini~\cite{Alessandrini-singular}, $W^{1,p}$ ($p > n$) by Nachman~\cite{Nachman1996GlobalUniqueness}, and $C^0$ by Sylvester and Uhlmann~\cite{Sylvester-Uhlmann-psiDO-bdry-det}. We have proved a low regularity exterior uniqueness result in Theorem~\ref{thm: exterior determination} for the fractional Calderón problem. This is analogous to the boundary determination result in~\cite{Brown}. Furthermore, it generalizes the result in \cite[Theorem 1.2]{RGZ2022GlobalUniqueness}.

It has been conjectured by Brown that the classical Calderón problem is uniquely solvable for uniformly elliptic conductivities in $W^{1,n}$ whenever $n \geq 3$:
\begin{conjecture}[{Brown~\cite[p. 565]{Brown-Torres03}}] Let $n \geq 3$ and $\Omega \subset \R^n$ be a bounded Lipschitz domain. Suppose that $\gamma_1,\gamma_2 \in W^{1,n}(\Omega)$ are uniformly elliptic. Then $\Lambda_{\gamma_1}=\Lambda_{\gamma_2}$ if and only if $\gamma_1=\gamma_2$.\label{conjecture: Brown}
\end{conjecture}

\begin{remark} In~\cite{Brown-Torres03}, Brown actually conjectures a slightly weaker statement: uniqueness for the inverse conductivity problem holds for the uniformly elliptic conductivities in $W^{1,p}(\Omega)$ when $p > n$. We state Conjecture \ref{conjecture: Brown} as it is due to Theorem \ref{thm: Haberman} of Haberman. We also mention here that Uhlmann had earlier conjectured uniqueness for the Lipschitz conductivities and this conjecture was eventually solved in the work of Caro and Rogers~\cite{CR16}. 
\end{remark}

Brown's conjecture is still open when $n \geq 5$, to the best of our knowledge. We have proved in Theorem~\ref{thm: Global uniqueness} that a fractional counterpart of Brown's conjecture holds in all dimensions $n \geq 1$. Finally, we recall two important achievements in this research area. The theorem of Astala and Päivärinta considers the low regularity inverse conductivity problem when $n=2$. See also \cite{AstalaLassasPaivarinta11, NRT2020NonlinearPlancherel} for other advances in the low regularity theory of the classical Calderón problem when $n=2$. The theorem of Haberman shows that Brown's conjecture holds when $n=3,4$.
\begin{theorem}[Astala--Päivärinta~\cite{AP06}]\label{thm: astala-paivarinta} Let $\Omega \subset \R^2$ be a bounded, simply connected domain. Suppose that $\gamma_1,\gamma_2 \in L^{\infty}(\Omega)$ are uniformly elliptic. Then $\Lambda_{\gamma_1}=\Lambda_{\gamma_2}$ if and only if $\gamma_1=\gamma_2$. 
\end{theorem} 
\begin{theorem}[Haberman~\cite{H15}] Let $n = 3,4$ and $\Omega \subset \R^n$ be a bounded Lipschitz domain. Suppose that $\gamma_1,\gamma_2 \in W^{1,n}(\Omega)$ are uniformly elliptic. Then $\Lambda_{\gamma_1}=\Lambda_{\gamma_2}$ if and only if $\gamma_1=\gamma_2$.\label{thm: Haberman}
\end{theorem}

\subsection{Main results}
We present and discuss our main theorems next. Theorem \ref{thm: Global uniqueness} generalizes the main result in \cite[Theorem 1.3]{RGZ2022GlobalUniqueness} from $H^{2s,\frac{n}{2s}} \cap H^s$ regularity to $H^{s,n/s}$ regularity. These results are not contained into each other and there is a tradeoff between differentiability and integrability. We have developed theory for the conductivities having $H^{s,n/s}$ regular background deviations for two reasons: 
\begin{itemize}
    \item This regularity permits to construct new counterexamples to partial data problems when $n=2,3$ for the cases missing in the earlier literature due to integrability issues.
    \item This regularity assumption is analogous to the classical Calderón problem for $W^{1,n}$ conductivities (see the previous discussion about Brown's conjecture).
\end{itemize} We believe that the $H^{s,n/s}$ regularity assumption cannot be \emph{substantially} improved without coming up with a new proof strategy (see also Question \ref{question: Linfty case} and later discussions). In general, one \emph{might} be able to develop theory for all conductivities having their background deviations in the interpolation spaces between $H^{2s,\frac{n}{2s}}$ and $H^{s,n/s}$ but this is out of our scope here. The fractional Liouville reduction from the fractional conductivity equation to the fractional Schrödinger equation breaks down with weaker Bessel regularity assumptions. Our theory in this article achieves the best Bessel regularity assumptions permitted by the current method of reducing the problem into an inverse problem for the fractional Schrödinger equation.

Our main theorem is as follows:
\begin{theorem}[Global uniqueness]
\label{thm: Global uniqueness}
    Let $\Omega\subset \R^n$ be an open set which is bounded in one direction and $0<s<\min(1,n/2)$. Assume that $\gamma_1,\gamma_2\in L^{\infty}(\R^n)$ are uniformly elliptic with $m_1,m_2 \in H^{s,n/s}(\R^n)$. Suppose that $W \subset \Omega_e$ is a nonempty open set such that $\gamma_1,\gamma_2$ are continuous a.e. in $W$.
Then $\left.\Lambda_{\gamma_1}f\right|_{W}=\left.\Lambda_{\gamma_2}f\right|_{W}$ for all $f\in C_c^{\infty}(W)$ if and only if $\gamma_1=\gamma_2$ in $\R^n$.
\end{theorem}

The assumption $H^{s,n/s}(\R^n)$ is needed for the Liouville reduction from the fractional conductivity equation to the fractional Schrödinger equation. We use Bessel potential spaces $H^{s,p}$ to model conductivities rather than the fractional Sobolev spaces $W^{s,p}$ since the fractional Laplacian has in these spaces strong mapping properties, i.e. $(-\Delta)^s\colon H^{t,p}(\R^n) \to H^{t-2s,p}(\R^n)$ is bounded. We assume that $\gamma_i$, $i=1,2$, have representatives, which are continuous a.e. in $W$ due to the exterior determination method (see Theorem~\ref{thm: exterior determination}).

We state the existence of counterexamples for general disjoint measurements in our next theorem. This extends \cite[Theorem 1.3]{RZ2022counterexamples} to cover all cases of $s \in (0,1)$ when $n\geq2$. The earlier theory developed in \cite{RGZ2022GlobalUniqueness,RZ2022unboundedFracCald,RZ2022counterexamples} had integrability issues in the construction of counterexamples when $s \in (n/4,1)$, $n=2,3$, as the background deviations had to be $H^{2s,\frac{n}{2s}}$ regular.

\begin{theorem}[Counterexamples]\label{thm: counterexample}
    Let $\Omega\subset \R^n$ be an open set which is bounded in one direction, $0<s<\min(1,n/2)$. For any nonempty open disjoint sets $W_1,W_2\subset\Omega_e$ with $\dist(W_1 \cup W_2, \Omega) > 0$ there exist two different conductivities $\gamma_1,\gamma_2\in L^{\infty}(\R^n) \cap C^\infty(\R^n)$ such that $\gamma_1(x),\gamma_2(x)\geq \gamma_0>0$, $m_1,m_2\in H^{s,n/s}(\R^n) \cap H^{s}(\R^n)$, and $\left.\Lambda_{\gamma_1}f\right|_{W_2}=\left.\Lambda_{\gamma_2}f\right|_{W_2}$ for all $f\in C_c^{\infty}(W_1)$.
\end{theorem}

The proofs of Theorems \ref{thm: Global uniqueness} and \ref{thm: counterexample} are given in Sections \ref{sec:lowregUniqueness} and \ref{sec: counterexamples}, respectively. 

\subsection{On the contributions of this work} The proofs of Theorems \ref{thm: Global uniqueness} and \ref{thm: counterexample} require that we analyze the two problems at hand individually, contrary to the earlier works \cite{RGZ2022GlobalUniqueness,RZ2022unboundedFracCald,RZ2022counterexamples} where the regularity assumptions were such that the uniqueness and nonuniqueness results had almost shared proofs. In the construction of counterexamples, our argument still requires one to work under the assumption that the background deviations are $H^s$ regular but in the proof of uniqueness, one of the merits in our present work is to get rid of this assumption.

We can remove the assumption that the background deviations belong to $H^s$ by a unique continuation property (UCP) for the fractional Laplacians in $H^{s,p}$ proved recently by Kar and the authors in \cite{KRZ2022Biharm}. The proof of Theorem \ref{thm:UCP-general} is based on $L^2$ Carleman estimates of Rüland with additional $L^p$ estimates and a localization argument for the Caffarelli--Silvestre extensions \cite{CS-nonlinera-equations-fractional-laplacians,CS-extension-problem-fractional-laplacian,KRZ2022Biharm,Ru15}. One also needs analytic regularity theory for elliptic equations and a certain iteration argument for the higher order cases \cite{CMR20,GSU20,KRZ2022Biharm}. The following holds:

\begin{theorem}[{UCP in Bessel potential spaces~\cite[Theorem~2.2]{KRZ2022Biharm}}]
\label{thm:UCP-general} Let $V \subset \R^n$ be a nonempty open set, $s \in \R_+ \setminus \N$, $t \in \R$ and $1 \leq p < \infty$. If $u \in H^{t,p}(\R^n)$ and $u|_V = (-\Delta)^s u|_V = 0$,
then $u \equiv 0$.
\end{theorem}
We remark that the UCP in the range $1\leq p\leq 2$ directly follows from the UCP for functions $u$ belonging to $H^t(\R^n)$, $t\in\R$, and the embedding theorems in Bessel potential spaces (see \cite[Corollary~3.5]{CMR20}).

We improve the assumption from $H^{2s,\frac{n}{2s}}$ to $H^{s,n/s}$ using sharp multiplication results in Bessel potential spaces, the Runst--Sickel lemmas, among a few other useful estimates in Bessel potential spaces. In this analysis, we are required to use Sobolev multipliers in the fractional Liouville reduction rather than $L^p$ functions as in the previous literature. However, the analysis of the Calderón problem for the fractional Schrödinger equations with Sobolev multiplier perturbations, the so called \emph{singular potentials}, was already established by Rüland and Salo in \cite{RS-fractional-calderon-low-regularity-stability}. Most notably one does not have to use the complex geometric optics solutions. The fractional Liouville reduction in low regularity shares many similarities with Brown's work on the classical Calderón problem \cite{Brown1996GlobalUniqueness}, e.g. both deal with Sobolev multipliers, multiplication estimates and global equations. There are some differences too, only in the fractional case one has to deal with singular integrals related to nonlocal operators. One also has to make certain weak$^*$ approximations and use many estimates for the fractional order Bessel potential spaces. Nevertheless, the Sobolev extensions of conductivities are not encountered in our considerations like in the construction of Brown \cite{Brown1996GlobalUniqueness}.

We also give a new strategy of proof for the uniqueness result in Section \ref{sec: second proof of main theorem} (see the second proof of Theorem \ref{thm: Global uniqueness}). The proof is based on a special relation of solutions for two conductivity equations with the same exterior condition when the DN maps agree (see Lemma \ref{lemma: relation for solutions}), the UCP, and the Alessandrini identity. This argument uses the UCP twice like the earlier proof but does not rely on the Runge approximation property in any step. We believe that this method of proof and Lemma \ref{lemma: relation for solutions} are of independent interest and may be useful in other problems.

Finally, the sharper regularity assumptions in the exterior determination method are based on a more direct proof than the one given in \cite{RGZ2022GlobalUniqueness}. In particular, we avoid using the fractional Liouville reduction completely in the proof of Theorem \ref{thm: exterior determination}. The proof of exterior determination is otherwise similar to \cite{RGZ2022GlobalUniqueness}.

\subsection{Further discussion and open problems}
Our first questions concerns what happens when one takes the limit $s \uparrow 1$:

\begin{question}\label{question: limiting case} Is it possible to obtain new insight about the classical Calderón problem using the analysis of the inverse fractional conductivity problem?
\end{question}
There is one mathematical necessity in such considerations: One should use or approximate exterior conditions, which do not vanish near $\partial \Omega$. If the analysis is only based on exterior conditions which vanish on $\partial \Omega$ (like often is the case in different kind of nonlocal Calderón problems), in the limit $s \uparrow 1$ such solutions cannot contain useful information of $\gamma$ in $\Omega$. This can be seen as in the limit one obtains, at least heuristically speaking, a solution to the associated conductivity equation, which is $0$ in $\Omega$ as the conductivity equation is local and the exterior condition vanishes on $\partial \Omega$. However, these approximations may get more informative and interesting when one uses exterior conditions, which do not vanish identically on $\partial \Omega$. Therefore, Question \ref{question: limiting case} is not only interesting itself but may additionally give new light to the understanding of the classical Calderón problem.

It might be possible to improve our assumptions in exterior determination (Theorem~\ref{thm: exterior determination}) by making more precise $L^p$ estimates and constructing suitable special solutions. It remains an open problem if it suffices to only use some Lebesgue point conditions for the conductivities and the Lebesgue differentiation theorem (or similar concepts). We believe that the positive answer to the following question could be true:
\begin{question}
\label{question: exterior determination Lebesgue points} Let $\Omega\subset \R^n$ be an open set which is bounded in one direction and $0<s<\min(1,n/2)$. Assume that $\gamma_1,\gamma_2\in L^{\infty}(\R^n)$ are uniformly elliptic and $W \subset \Omega_e$ is an open set. Does $\left.\Lambda_{\gamma_1}f\right|_{W}=\left.\Lambda_{\gamma_2}f\right|_{W}$ for all $f\in C_c^{\infty}(W)$ imply that $\gamma_1=\gamma_2$ a.e. in $W$?
\end{question}

Our two main results motivate the following two problems, which we have not been able to answer despite of some efforts:
\begin{question}[Fractional Astala--Päivärinta theorem]
\label{question: Linfty case} 
Does the partial/full data uniqueness hold for the uniformly elliptic conductivities that only satisfy $\gamma \in L^\infty(\R^n)$? Can one remove the assumption that conductivities converge to the trivial conductivity at infinity by some other regularity assumptions?
\end{question}

The positive answer to the first question would be analogous to Theorem \ref{thm: astala-paivarinta} of Astala and Päivärinta for the classical Calderón problem in two dimensions. The other question of a sufficient/necessary decay at infinity is open as long as one does not assume that $\gamma$ is from a generic class like the class of real analytic conductivities. Such genericity assumptions ''localize'' the problem into the exterior determination problem which can be uniquely solved. We further note that the low regularity theory of the inverse fractional conductivity problem in all dimensions, as developed in \cite{RGZ2022GlobalUniqueness,C21,RZ2022unboundedFracCald} and here, resembles more the higher dimensional than the two dimensional classical Calderón problem. Only little is known about uniqueness or nonuniqueness for the classical Calderón problem with uniformly elliptic $L^\infty$ conductivities when $n \geq 3$, see e.g.~\cite{Santacesaria2019} and the references therein for a more detailed discussion.

\begin{question}
\label{question: all counterexamples} 
Are there counterexamples to uniqueness in the partial data inverse problem for all nonempty open sets $W_1,W_2 \subset \R^n$ such that $W_1 \cap W_2 =\emptyset$ and $\dist(W_1 \cup W_2,\Omega)=0$?
\end{question}
These counterexamples would be new and interesting with any regularity assumption. We do not know the answer to Question \ref{question: all counterexamples} even when $\Omega \subset \R^n$ is assumed to be a bounded open set or $n=1$. The complete understanding of Question \ref{question: all counterexamples} would imply the full characterization of uniqueness for the inverse fractional conductivity problem with partial data (in the given regularity) in combination with Theorems \ref{thm: Global uniqueness} and \ref{thm: counterexample}.

\begin{figure}[!ht]
    \centering
    \scalebox{0.85}{
    \begin{tikzpicture}
    \fill[cyan!5, xshift=4cm,yshift=2.5cm] (-1,0.25)-- (12.5,0.25) --(12.5,-1.25)--(-1,-1.25)--(-1,0.25);
    \fill[green!5, xshift=4cm,yshift=2.5cm] (-1,-1.25)-- (12.5,-1.25) --(12.5,-2.75)--(-1,-2.75)--(-1,-1.25);
    \filldraw[color=yellow!50, fill=yellow!10, xshift=7.5cm](2,1.25) circle (1);
    \node[xshift=3.5cm] at (2,2)
    {$\raisebox{-.35\baselineskip}{\Large\ensuremath{W_1}}$};
    \node[xshift=3.5cm] at (2,0.5)
    {$\raisebox{-.35\baselineskip}{\Large\ensuremath{W_2}}$};
    \node[xshift=7.5cm, yshift=0.4cm] at (2.35,1.25)
    {$\raisebox{-.35\baselineskip}{\Large\ensuremath{\Omega}}$};
    \draw [gray,-{Stealth}] (2.9,1.25) --(16.8,1.25);
    \draw [gray,-{Stealth}] (9.5,1.15) --(9.5,3);
    \node at (17.1,1.25)
    {$\raisebox{-.35\baselineskip}{\ensuremath{x_1}}$};
    \node at (9.5,3.3)
    {$\raisebox{-.35\baselineskip}{\ensuremath{x_2}}$};
\end{tikzpicture}}
    \caption{An example illustrating a prototypical situation to be considered in Question~\ref{question: all counterexamples}. The major difficulty in such geometrical settings is that one does not have enough room to use the method of mollification (see \cite{RZ2022counterexamples}) to guarantee the required regularity conditions for the background deviations $m_i\vcentcolon = \gamma_i^{1/2}-1$ for $i=1,2$.}
    \label{fig: Runge Approximation}
\end{figure}
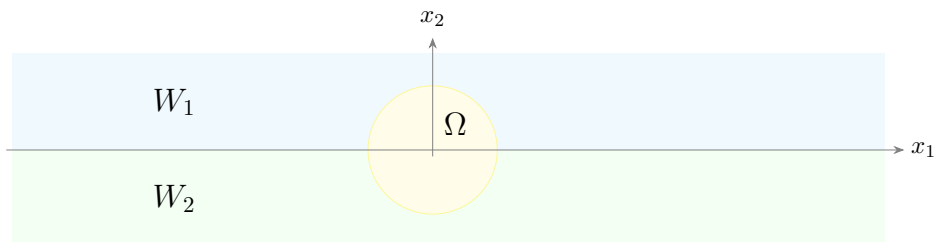

\subsection*{Acknowledgements} The authors thank Giovanni Covi and Joaquim Serra for helpful discussions related to our work. J.R. was supported by the Vilho, Yrjö and Kalle Väisälä Foundation of the Finnish Academy of Science and Letters.

\section{Preliminaries}\label{sec:preliminaries}

\subsection{Basic notation}

 Throughout this article the space of Schwartz functions and the space of tempered distributions will be denoted by $\schwartz(\R^n)$ and $\tempered(\R^n)$, repsectively. The Fourier transform of $u\in\schwartz(\R^n)$ is defined as
\[
    \fourier u(\xi)\vcentcolon = \hat u(\xi) \vcentcolon = \int_{\R^n} u(x)e^{-ix \cdot \xi} \,dx.
\]
Moreover, the Fourier transform acts as an isomorphism on the spaces $\schwartz(\R^n)$, $\tempered(\R^n)$ and we denote the inverse of the Fourier transform by $\ifourier$ in each case. The Bessel potential of order $s \in \R$ is the Fourier multiplier $\vev{D}^s\colon \tempered(\R^n) \to \tempered(\R^n)$, that is
\begin{equation}\label{eq: Bessel pot}
    \vev{D}^s u \vcentcolon = \ifourier(\vev{\xi}^s\widehat{u}),
\end{equation} 
where $\vev{\xi}\vcentcolon = (1+|\xi|^2)^{1/2}$. For any $s \in \R$ and $1 \leq p < \infty$, the Bessel potential space $H^{s,p}(\R^n)$ is defined by
\begin{equation}
\label{eq: Bessel pot spaces}
    H^{s,p}(\R^n) \vcentcolon = \{ u \in \tempered(\R^n)\,;\, \vev{D}^su \in L^p(\R^n)\},
\end{equation}
 which we endow with the norm $\norm{u}_{H^{s,p}(\R^n)} \vcentcolon = \norm{\vev{D}^su}_{L^p(\R^n)}$. If $\Omega\subset \R^n$, $F\subset\R^n$ are given open and closed sets, then we define the following local Bessel potential spaces:
\begin{equation}\label{eq: local bessel pot spaces}
\begin{split}
    \widetilde{H}^{s,p}(\Omega) &\vcentcolon = \mbox{closure of } C_c^\infty(\Omega) \mbox{ in } H^{s,p}(\R^n),\\
    H_F^{s,p}(\R^n) &\vcentcolon =\{\,u \in H^{s,p}(\R^n)\,;\, \supp(u) \subset F\,\}.
\end{split}
\end{equation}
We see that $\widetilde{H}^{s,p}(\Omega),H^{s,p}_F(\R^n)$ are closed subspaces of $H^{s,p}(\R^n)$. As customary, we omit the index $p$ from the above notations in the case $p=2$.

If $u\in\tempered(\R^n)$ is a tempered distribution and $s\geq 0$, the fractional Laplacian of order $s$ of $u$ is the Fourier multiplier
\[
    (-\Delta)^su\vcentcolon = \ifourier(|\xi|^{2s}\widehat{u}),
\]
whenever the right hand side is well-defined. If $p\geq 1$ and $t\in\R$, the fractional Laplacian is a bounded linear operator $(-\Delta)^{s}\colon H^{t,p}(\R^n) \to H^{t-2s,p}(\R^n)$. In the special case $u\in\schwartz(\R^n)$ and $s\in(0,1)$, we have the identities (see e.g.~\cite[Section 3]{DINEPV-hitchhiker-sobolev})
\begin{equation}
\begin{split}
    (-\Delta)^su(x)&=C_{n,s}\,\text{p.v.}\int_{\R^n}\frac{u(x)-u(y)}{|x-y|^{n+2s}}\,dy\\
    &= -\frac{C_{n,s}}{2}\int_{\R^n}\frac{u(x+y)+u(x-y)-2u(x)}{|y|^{n+2s}}\,dy,
\end{split}
\end{equation}
where $C_{n,s}\vcentcolon =\frac{4^s\Gamma(n/2+s)}{\pi^{n/2}|\Gamma(-s)|}$. Moreover, in the study of the inverse problem related to the fractional conductivity equation one property of the fractional Laplacian turns out to be essential, namely the UCP (see Theorem~\ref{thm:UCP-general}).

Another important property of the fractional Laplacian which allows to study the fractional conductivity equation on domains bounded in one direction (cf.~\cite[Definition 2.1]{RZ2022unboundedFracCald}) is the Poincar\'e inequality, which says that for any $\Omega\subset\R^n$ bounded in one direction there exists a constant $C>0$ such that
\[
    \|u\|_{L^p(\R^n)}\leq C\|(-\Delta)^{s/2}u\|_{L^p(\R^n)}
\]
for all $u\in C_c^{\infty}(\Omega)$, where $s\geq 0$, $p\geq 2$ or $s\geq 1$, $1<p<2$ (cf.~\cite[Theorem 2.2]{RZ2022unboundedFracCald}). The Poincar\'e inequality in the range $0<s<1$, $1<p<2$, for unbounded domains is not present in the existing literature to the best of our knowledge.
    
For the rest of this section we fix $s\in(0,1)$. The fractional gradient of order $s$ is the bounded linear operator $\nabla^s\colon H^s(\R^n)\to L^2(\R^{2n};\R^n)$ given by (see ~\cite{C21,DGLZ12,RZ2022unboundedFracCald})
    \[
        \nabla^su(x,y)\vcentcolon =\sqrt{\frac{C_{n,s}}{2}}\frac{u(x)-u(y)}{|x-y|^{n/2+s+1}}(x-y),
    \]
    and satisfies
    \begin{equation}
    \label{eq: bound on fractional gradient}
        \|\nabla^su\|_{L^2(\R^{2n})}=\|(-\Delta)^{s/2}u\|_{L^2(\R^n)}\leq \|u\|_{H^s(\R^n)}
    \end{equation}
    for all $u\in H^s(\R^n)$. The adjoint of $\nabla^s$ is called fractional divergence of order $s$ and denoted by $\Div_s$. More concretely, the fractional divergence of order $s$ is the bounded linear operator 
    \[
        \Div_s\colon L^2(\R^{2n};\R^n)\to H^{-s}(\R^n)
    \] 
    satisfying
    \[
        \langle \Div_su,v\rangle_{H^{-s}(\R^n)\times H^s(\R^n)}=\langle u,\nabla^sv\rangle_{L^2(\R^{2n})}
    \]
    for all $u\in L^2(\R^{2n};\R^n),v\in H^s(\R^n)$. One can show that (see ~\cite[Section 8]{RZ2022unboundedFracCald})
    \[
        \|\Div_s(u)\|_{H^{-s}(\R^n)}\leq \|u\|_{L^2(\R^{2n})}
    \]
    for all $u\in L^2(\R^{2n};\R^n)$, and also $(-\Delta)^su=\Div_s(\nabla^su)$ weakly for all $u\in H^s(\R^n)$ (see ~\cite[Lemma 2.1]{C21}). 

\subsection{Triebel--Lizorkin spaces and the Runst--Sickel lemma}
Before recalling several important results from harmonic analysis which will be used throughout this article, we introduce the Triebel--Lizorkin spaces $F^s_{p,q}=F^s_{p,q}(\R^n)$ following the exposition in \cite{BrezisComposition} or \cite{Triebel-Theory-of-function-spaces,Triebel-Theory-of-functions-spaces-I-I}. In the sequel, we will write for brevity $B_r$ instead of $B_r(0)$. Fix any $\psi_0\in C_c^{\infty}(\R^n)$ satisfying
\[
    0\leq \psi_0\leq 1,\quad \psi_0(\xi)=1\quad\text{for}\quad |\xi|\leq 1\quad \text{and}\quad \psi_0(\xi)=0\quad \text{for}\quad |\xi|\geq 2.
\]
Let $\psi_j\in C_c^{\infty}(B_{2^{j+1}})$, $j\geq 1$, be given by
\[
    \psi_j(\xi)=\psi_0(\xi/2^j)-\psi_0(\xi/2^{j-1}).
\]
In this section, we write $u_j\vcentcolon =u\ast \phi_j=\ifourier(\psi_j\hat{u})\in \schwartz(\R^n)$ for any $u\in\tempered(\R^n)$, where $\phi_j=\ifourier(\psi_j)$ ($j\geq 1$) and one has the Littlewood--Paley decomposition
\[
    u=\sum_{j\geq 0}u_j\quad \text{in}\quad \tempered(\R^n).
\]
For all $s\in\R$ and $0<p,q\leq \infty$, we set (\cite[Section 2.3.1]{Triebel-Theory-of-function-spaces})
\[
    F^s_{p,q}=\{\,u\in\tempered(\R^n)\,;\, \|u\|_{F^s_{p,q}}=\|\,\|2^{js}u_j(x)\|_{\ell^q}\|_{L^p(\R^n)}<\infty\}.
\]
\begin{remark}
\label{remark: identification}
    \begin{enumerate}[(i)]
        \item\label{item 1 triebel spaces} For $0<p<\infty$ different choices of $\psi_0$ yield equivalent quasi-norms (see \cite[Section 2.3.5]{Triebel-Theory-of-function-spaces}), but for $p=\infty$, $0<q<\infty$ this is in general wrong as shown in \cite[Section 2.3.2]{Triebel-Theory-of-functions-spaces-I-I}, and for $s\in\R$, $0<p<\infty$, $0<q\leq \infty$ the Triebel-Lizorkin spaces are quasi-Banach spaces and Banach spaces if $p,q\geq 1$ (see \cite[Section 2.3.3]{Triebel-Theory-of-function-spaces}).
        \item\label{item 2 triebel spaces} By the embedding $\ell^{q_1}\hookrightarrow \ell^{q_2}$ for $0<q_1\leq q_2\leq \infty$, we have $F^s_{p,q_1}\subset F^s_{p,q_2}$ when $s\in\R$, $0<p\leq \infty$ and $0<q_1\leq q_2\leq \infty$.
        \item\label{item 3 triebel spaces} We have the following identifications with equivalent norms (\cite{BrezisComposition,Triebel-Theory-of-function-spaces}):
        \begin{enumerate}[(I)]
            \item\label{first identification remark A1} $F^0_{p,2}=L^p(\R^n)$ for $1<p<\infty$,
            \item\label{second identification remark A1} $F^s_{p,2}=H^{s,p}(\R^n)$ for $1<p<\infty$, $s\in\R$,
            \item\label{third identification remark A1} $F^s_{p,p}=W^{s,p}(\R^n)$ for $1\leq p<\infty$, $s\in (0,\infty)\setminus\N$,
            \item\label{fourth identification remark A1} $L^{\infty}(\R^n)\hookrightarrow F^{0}_{\infty,\infty}$ with
            \[
            \|u\|_{F^0_{\infty,\infty}}=\sup_{j\in\N_0,x\in\R^n}|u_j(x)|\leq C\|u\|_{L^{\infty}(\R^n)}.
            \]
        \end{enumerate}
    \end{enumerate}
\end{remark}

\begin{proposition}[{Runst-Sickel lemma I, \cite[Lemma 5]{BrezisComposition} or \cite[p. 345]{RunstSickel}}]
\label{prop: Runs-Sickel Lemma I}
    Let $0<s<\infty$, $1<q<\infty$, $1<p_1,p_2,r_1,r_2\leq \infty$ satisfy
    \[
        0<\frac{1}{p}\vcentcolon =\frac{1}{p_1}+\frac{1}{r_2}=\frac{1}{p_2}+\frac{1}{r_1}<1,
    \]
then for all $f\in F^s_{p_1,q}\cap L^{r_1}(\R^n)$, $g\in F^s_{p_2,q}\cap L^{r_2}(\R^n)$ there holds
\[
    \|fg\|_{F^{s}_{p,q}}\leq C(\|Mf(x)\|2^{sj}g_j(x)\|_{\ell^q}\|_{L^p(\R^n)}+\|Mg(x)\|2^{sj}f_j(x)\|_{\ell^q}\|_{L^p(\R^n)})
\]
and
\begin{equation}
\label{eq: estimate Runst Sickel}
    \|fg\|_{F^{s}_{p,q}}\leq C(\|f\|_{F^{s}_{p_1,q}}\|g\|_{L^{r_2}(\R^n)}+\|g\|_{F^{s}_{p_2,q}}\|f\|_{L^{r_1}(\R^n)}).
\end{equation}
\end{proposition}

\begin{remark}
    Note that the Kato--Ponce inequality for $1<p<\infty$ (see \cite[Theorem 1]{KatoPonceGrafakos}, \cite[Theorem 1.4]{SmoothingPropSchrOp}) can be seen as a special case of the estimate \eqref{eq: estimate Runst Sickel} by choosing $q=2$ and using the fact that $F^{s}_{p,2}=H^{s,p}$ (cf.~Remark~\ref{remark: identification}\ref{second identification remark A1}).
\end{remark}

\begin{proposition}[{Runst-Sickel lemma II, \cite[Corollary 3]{BrezisComposition}}]
\label{prop: Runs-Sickel Lemma II}
    Let $0<s<\infty$, $1<q<\infty$. Then the following assertions hold:
    \begin{enumerate}[(i)]
        \item\label{item 1 runst sickel II } If $1<p_1,p_2,r_1,r_2< \infty$ satisfy
    \[
        0<\frac{1}{p}\vcentcolon =\frac{1}{p_1}+\frac{1}{r_2}=\frac{1}{p_2}+\frac{1}{r_1}<1,
    \] then the multiplication map 
        \[
            (F^{s}_{p_1,q}\cap L^{r_1}(\R^n))\times (F^{s}_{p_2,q}\cap L^{r_2}(\R^n))\ni (f,g)\mapsto fg\in F^s_{p,q}
        \]
        is continuous.
        \item\label{item 2 runst sickel II } If $1<p<\infty$ and there holds
        \[
            \begin{cases}
                    f^k\to f\quad \text{in}\quad F^s_{p_1,q},\, & \|f^k\|_{L^{\infty}(\R^n)}\leq C,\\
                    g^k\to g\quad \text{in}\quad F^s_{p_1,q},\, & \|g^k\|_{L^{\infty}(\R^n)}\leq C
            \end{cases}
        \]
        for some $C>0$, then $f^kg^k\to fg$ in $F^{s}_{p,q}$.
        \item\label{item 3 runst sickel II } Let $1<p_1,r,p<\infty$ be such that
        \[
            \frac{1}{p}=\frac{1}{p_1}+\frac{1}{r}
        \]
        and there holds
        \[
            \begin{cases}
                    f^k\to f\quad \text{in}\quad F^s_{p_1,q},\, & \|f^k\|_{L^{\infty}(\R^n)}\leq C,\\
                    g^k\to g\quad \text{in}\quad F^s_{p,q}\cap L^r\, &
            \end{cases}
        \]
        for some $C>0$, then $f^kg^k\to fg$ in $F^{s}_{p,q}$.
    \end{enumerate}
\end{proposition}

\subsection{Sobolev multipliers}

We recall now the definition of Sobolev multipliers. We mainly follow the exposition in \cite{Mazya}. Applications to inverse problems can be found in \cite{RS-fractional-calderon-low-regularity-stability,CMRU22,RZ2022unboundedFracCald}. We let $\distr(\R^n)$ stand for the space of distributions and by $\cdistr(\R^n)$ the space of compactly supported distributions.  For any $s\in \R$, we denote by $M(H^s\rightarrow H^{-s})$ the space of Sobolev multipliers of order $s$, which consists of all distributions $f\in \distr(\R^n)$ such that
$$\|f\|_{s,-s} \vcentcolon = \sup \{\abs{\ip{f}{u v}} \,;\, u,v \in C_c^\infty(\mathbb R^n), \norm{u}_{H^s(\R^n)} = \norm{v}_{H^{s}(\R^n)} =1 \}$$
is finite. Moreover, we let $M_0(H^s \to H^{-s})$ be the closure of $C_c^\infty(\R^n)$ in $M(H^s \to H^{-s})$. According to the authors knowledge it is still an open problem to fully characterize $M_0(H^s\rightarrow H^{-s})$ and to show that $M_0(H^s\rightarrow H^{-s})\subsetneq M(H^s\rightarrow H^{-s})$. A partial answer to this problem is given in \cite[Remark 2.5]{RS-fractional-calderon-low-regularity-stability}, which states that $C_c^{\infty}(\R^n)$ is dense in $M(H^{s-\delta}\to H^{-s+\delta})\cap \cdistr(\R^n)$ if the later space is endowed with the norm inherited from $M(H^s\to H^{-s})$.  If $f\in M(H^s\rightarrow H^{-s})$ and $u,v \in C_c^\infty(\mathbb R^n)$ are both nonvanishing, we have the multiplier inequality
\begin{equation}\label{multiplier-inequality}
    \abs{\langle f,uv\rangle} \leq \|f\|_{s,-s}\norm{u}_{H^s(\R^n)} \norm{v}_{H^{s}(\R^n)}.
\end{equation}

By density of $C_c^\infty(\R^n) \times C_c^\infty(\R^n)$ in $ H^s(\R^n) \times H^{s}(\R^n)$, there is a unique continuous extension $(u,v) \mapsto \langle f, uv \rangle$ for $(u,v)\in H^s(\R^n)\times H^{s}(\R^n)$. More precisely, each $f \in M(H^s\rightarrow H^{-s})$ gives rise to a linear multiplication map $m_f \colon H^s(\R^n) \rightarrow H^{-s}(\R^n)$ defined by 
\begin{equation}
\label{eq: multiplication by sobolev multiplier}
    \langle m_f(u),v \rangle \vcentcolon = \lim_{i \to \infty}\langle f,u_iv_i \rangle \quad \mbox{for all} \quad (u,v)\in H^s(\R^n)\times H^{s}(\R^n),
\end{equation}
where $(u_i,v_i) \in C_c^\infty(\R^n) \times C_c^\infty(\R^n)$ is any sequence in $H^s(\R^n) \times H^{s}(\R^n)$ converging to $(u,v)$. We will just write $fu$ instead of $m_f(u)$.

\subsection{Direct problem}
\begin{lemma}[{Definition of bilinear forms and conductivity matrix (cf.~\cite[Lemma 8.3]{RZ2022unboundedFracCald})}]\label{prop: bilinear forms conductivity eq}
    Let $\Omega\subset\R^n$ be an open set, $0<s<1$, $\gamma\in L^{\infty}(\R^n)$ and define the conductivity matrix associated to $\gamma$ by
    \begin{equation}
    \label{eq: conductivity matrix}
        \Theta_{\gamma}\colon \R^{2n}\to \R^{n\times n},\quad \Theta_{\gamma}(x,y)\vcentcolon =\gamma^{1/2}(x)\gamma^{1/2}(y)\mathbf{1}_{n\times n}
    \end{equation}
    for $x,y\in\R^n$. Then the map defined by
    \begin{equation}
    \label{eq: conductivity bilinear form}
        B_{\gamma}\colon H^s(\R^n)\times H^s(\R^n)\to \R,\quad B_{\gamma}(u,v)\vcentcolon =\int_{\R^{2n}}\Theta_{\gamma}\nabla^su\cdot\nabla^sv\,dxdy
    \end{equation}
    is continuous bilinear form. 
\end{lemma}
\begin{remark}
    If no confusion can arise, we will drop the subscript $\gamma$ in the definition for the conductivity matrix $\Theta_{\gamma}$.
\end{remark}

\begin{definition}[Weak solutions]
    Let $\Omega\subset\R^n$ be an open set, $0<s<1$ and $\gamma\in L^{\infty}(\R^n)$ with conductivity matrix $\Theta\colon \R^{2n}\to \R^{n\times n}$. If $f\in H^s(\R^n)$ and $F\in (\widetilde{H}^s(\Omega))^*$, then we say that $u\in H^s(\R^n)$ is a weak solution to the fractional conductivity equation
     \[
     \begin{split}
        \Div_s(\Theta\nabla^s u)&= F\quad\text{in}\quad\Omega,\\
        u&= f\quad\text{in}\quad\Omega_e
     \end{split}
     \]
    if there holds
    \[
        B_{\gamma}(u,\phi)=F(\phi)\quad\text{and}\quad u-f\in\widetilde{H}^s(\Omega)
    \]
    for all $\phi\in \widetilde{H}^s(\Omega)$. 
\end{definition}

\begin{lemma}[{Well-posedness and DN maps (cf.~\cite[Lemma 8.10]{RZ2022unboundedFracCald})}]
\label{lemma: well-posedness results and DN maps}
    Let $\Omega\subset \R^n$ be an open set which is bounded in one direction and $0<s<1$. Assume that $\gamma\in L^{\infty}(\R^n)$ with conductivity matrix $\Theta$ satisfies $\gamma \geq \gamma_0 > 0$. Then the following assertions hold:
    \begin{enumerate}[(i)]
        \item\label{item 1 well-posedness cond eq} For all $f\in X\vcentcolon = H^s(\R^n)/\widetilde{H}^s(\Omega)$ there is a unique weak solution $u_f\in H^s(\R^n)$ of the fractional conductivity equation
        \begin{equation}
        \begin{split}
            \Div_s(\Theta\nabla^s u)&= 0\quad\text{in}\quad\Omega,\\
            u&= f\quad\text{in}\quad\Omega_e.
        \end{split}
        \end{equation}
        \item\label{item 2 well-posedness cond eq} The exterior DN map $\Lambda_{\gamma}\colon X\to X^*$ given by 
        \[
        \begin{split}
            \langle \Lambda_{\gamma}f,g\rangle \vcentcolon =B_{\gamma}(u_f,g),
        \end{split}
        \]
        where $u_f\in H^s(\R^n)$ is the unique solution to the fractional conductivity equation with exterior value $f$, is a well-defined bounded linear map. 
    \end{enumerate}
\end{lemma}

\begin{remark}
    We note that the additional condition $m\in H^{2s,\frac{n}{2s}}(\R^n)$ in \cite[Lemma 8.10]{RZ2022unboundedFracCald} were only needed to establish the related assertions for the fractional Schr\"odinger equation.
\end{remark}

\section{Low regularity uniqueness for the inverse fractional conductivity problem}
\label{sec:lowregUniqueness}

\subsection{Low regularity exterior determination} We need to establish an exterior determination result for the fractional Calderón problem. We prove a generalization of \cite[Theorem 1.2]{RGZ2022GlobalUniqueness} without any Sobolev regularity for the conductivities. We will be very brief in the proof since the proof of \cite[Theorem 1.2]{RGZ2022GlobalUniqueness} holds almost identically. The only concern is to check that the required elliptic energy estimates (see~\cite[Corollary 5.4]{RGZ2022GlobalUniqueness}) are still valid for the related exterior value problems. We also remark that one does not have to reduce this estimate to the similar property of Schrödinger equations (via the Liouville reduction) as done earlier in \cite{RGZ2022GlobalUniqueness}.

\begin{lemma}
\label{lemma: elliptic estimate}
Let $0<s<1$. Suppose that $\Omega\subset\R^n$ is an open set which is bounded in one direction, $W\subset\Omega_e$ an open set with finite measure and positive distance from $\Omega$. Assume that $\gamma\in L^{\infty}(\R^n)$ with conductivity matrix $\Theta$ satisfies $\gamma(x)\geq \gamma_0>0$. Then for any $f\in C_c^{\infty}(W)$ the associated unique solution $u_f\in H^s(\R^n)$ of
    \begin{equation}
    \label{eq: conductivity equation elliptic estimate}
     \begin{split}
        \Div_s(\Theta\nabla^s u)&= 0\quad\text{in}\quad\Omega,\\
        u&= f\quad\text{in}\quad\Omega_e
     \end{split}
     \end{equation}
     satisfies the estimate
     \[
        \|u_f-f\|_{H^s(\R^n)}\leq C\|f\|_{L^2(W)}
     \]
     for some $C>0$ depending only on $n,s,\gamma_0,\|\gamma\|_{L^{\infty}(\Omega\cup W)},|W|$, the distance between $W$ and $\Omega$, and the Poincar\'e constant of $\Omega$.
\end{lemma}
\begin{proof}
    By~\cite[Lemma 8.10]{RZ2022unboundedFracCald} the unique solution $u_f\in H^s(\R^n)$ to \eqref{eq: conductivity equation elliptic estimate} satisfies the estimate
    \[
        \|u_f-f\|_{H^s(\R^n)}\leq C\|B_{\gamma}(f,\cdot)\|_{(\widetilde{H}^s(\Omega))^*}.
    \]
    Therefore, by the definition of $B_{\gamma}$, $\supp(f)\subset\Omega_e$ and the support argument carried out in detail in \cite[Proof of Lemma 5.3]{RGZ2022GlobalUniqueness} we have
    \[
    \begin{split}
        |B_{\gamma}(f,\phi)|&=\frac{C_{n,s}}{2}\left|\int_{\R^{2n}}\gamma^{1/2}(x)\gamma^{1/2}(y)\frac{(f(x)-f(y))(\phi(x)-\phi(y))}{|x-y|^{n+2s}}\,dxdy\right|\\
        &=C_{n,s}\left|\int_{W\times\Omega}\gamma^{1/2}(x)\gamma^{1/2}(y)\frac{f(y)\phi(x)}{|x-y|^{n+2s}}\,dxdy\right|\\
        &=C_{n,s}\left|\int_{W\times\Omega}\frac{(\gamma^{1/2}(y)f(y))(\gamma^{1/2}(x)\phi(x))}{|x-y|^{n+2s}}\,dxdy\right|.
    \end{split}
    \]
    for all $\phi\in \widetilde{H}^s(\Omega)$. 
    
    Following \cite[Proof of Lemma 5.3]{RGZ2022GlobalUniqueness}, we obtain by the Cauchy--Schwartz and the Minkowski inequality the estimate
    \[
    \begin{split}
        &|B_{\gamma}(f,\phi)|\leq  C_{n,s}\sqrt{\frac{\omega_n}{(n+4s)r^{n+4s}}}|W|^{1/2}\|\gamma^{1/2}f\|_{L^2(W)}\|\gamma^{1/2}\phi\|_{L^2(\Omega)}\\
        &\leq C_{n,s}\|\gamma\|^{1/2}_{L^{\infty}(\Omega)}\|\gamma\|_{L^{\infty}(W)}^{1/2}\sqrt{\frac{\omega_n}{(n+4s)r^{n+4s}}}|W|^{1/2}\|f\|_{L^2(W)}\|\phi\|_{L^2(\Omega)},
    \end{split}
    \]
    where $\omega_n$ denotes the Lebesgue measure of the unit ball. This shows
    \[
        \|u_f-f\|_{H^s(\R^n)}\leq CC_{n,s}\|\gamma\|^{1/2}_{L^{\infty}(\Omega)}\|\gamma\|_{L^{\infty}(W)}^{1/2}\sqrt{\frac{\omega_n}{(n+4s)r^{n+4s}}}|W|^{1/2}\|f\|_{L^2(W)}.
    \]
\end{proof}

\begin{theorem}[Exterior determination]
\label{thm: exterior determination}
    Let $\Omega\subset \R^n$ be an open set which is bounded in one direction and $0<s<1$. Assume that $\gamma_1,\gamma_2\in L^{\infty}(\R^n)$ satisfy $\gamma_1(x),\gamma_2(x)\geq \gamma_0>0$. 
    Suppose that $W \subset \Omega_e$ is a nonempty open set such that $\gamma_1,\gamma_2$ are continuous a.e. in $W$.
If $\left.\Lambda_{\gamma_1}f\right|_{W}=\left.\Lambda_{\gamma_2}f\right|_{W}$ for all $f\in C_c^{\infty}(W)$, then $\gamma_1=\gamma_2$ a.e. in $W$.
\end{theorem}
\begin{proof} 
    First choose a set $\mathcal{N}\subset W$ such that $|\mathcal{N}|=0$ and $\gamma_1,\gamma_2$ are continuous on $V\vcentcolon =W\setminus \mathcal{N}$. Let $x_0\in V$. Next note that \cite[Lemma 5.5]{RGZ2022GlobalUniqueness} holds for any $0<s<1$ and therefore by \cite[Remark 5.6]{RGZ2022GlobalUniqueness} there exists a sequence $(\phi_N)_{N\in\N}\subset C_c^{\infty}(W)$ such that 
    \begin{enumerate}[(i)]
        \item\label{item 1: normalization} $\|\phi_N\|_{L^2(\R^n)}^2+\|(-\Delta)^{s/2}\phi_N\|_{L^2(\R^n)}^2=1$ for all $N\in\N$, 
        \item \label{item 2: estimate} $\|\phi_N\|_{L^2(\R^n)}\to 0$ as $N\to\infty$,
        \item \label{item 3: support} $\supp(\phi_N)\to \{x_0\}$
        as $N\to\infty$. 
    \end{enumerate}
    
    By the product rule for the fractional gradient 
    \[
        \nabla^s(uv)(x,y)=v(y)\nabla^su(x,y)+u(x)\nabla^sv(x,y)
    \]
    for all $u,v\in H^s(\R^n)$ and \ref{item 2: estimate} one obtains (cf.~\cite[eq.~(18)]{RGZ2022GlobalUniqueness})
    \[
    \begin{split}
        &\limsup_{N\to\infty}\left|\int_{\R^{2n}}(\gamma^{1/2}_i(y)-\gamma_i^{1/2}(x_0))g(x)|\nabla^s\phi_N|^2\,dxdy\right|\\
        &\leq C\|g\|_{L^{\infty}(\R^n)}^{1/2}\|\gamma_i^{1/2}-\gamma_i^{1/2}(x_0)\|_{L^{\infty}(Q_{1/M}(x_0))}
    \end{split}
    \]
    for any bounded function $g\in L^{\infty}(\R^n)$, $M\in\N$ and $i=1,2$. Here the cube $Q_r(x)$ is given for all $x\in\R^n$, $r>0$ by
    \[
        Q_r(x)\vcentcolon =\{y\in\R^n;\,\,|y_j-x_j|<r\,\,\text{for all}\, j=1,\ldots,n\}.
    \] 
    
    By the assumptions the conductivities $\gamma_1,\gamma_2$ are continuous at $x_0$ and therefore passing to the limit $M\to\infty$ gives
    \begin{equation}
    \label{eq: limit}
        \lim_{N\to\infty}\int_{\R^{2n}}(\gamma^{1/2}_i(y)-\gamma_i^{1/2}(x_0))g(x)|\nabla^s\phi_N|^2\,dxdy=0
    \end{equation}
    for any $g\in L^{\infty}(\R^n)$ and $i=1,2$. This ensures
    \begin{equation}
    \label{eq: reconstruction}
    \begin{split}
        \gamma_i(x_0)
        =&\lim_{N\to\infty}\int_{\R^{2n}}\gamma_i^{1/2}(x)(\gamma_i^{1/2}(y)-\gamma_i^{1/2}(x_0))|\nabla^s\phi_N|^2\,dxdy\\
        &+\gamma_i^{1/2}(x_0)\lim_{N\to\infty}\int_{\R^{2n}}(\gamma_i^{1/2}(x)-\gamma_i^{1/2}(x_0))|\nabla^s\phi_N|^2\,dxdy\\
        &+\gamma_i(x_0)\lim_{N\to\infty}\int_{\R^{2n}}|\nabla^s\phi_N|^2\,dxdy\\
        =&\lim_{N\to\infty}\int_{\R^{2n}}\gamma_i^{1/2}(x)\gamma_i^{1/2}(y)|\nabla^s\phi_N|^2\,dxdy\\
       =&\lim_{N\to\infty}\langle \Theta_{\gamma_i}\nabla^s\phi_N,\nabla^s\phi_N\rangle_{L^2(\R^{2n})}\\
       =&\lim_{N\to\infty} E_{\gamma_i}(\phi_N)
    \end{split}
    \end{equation}
    for $i=1,2$. In the first equality sign, we used the equation \eqref{eq: limit} for $g=\gamma_i$ and $g=1$ as well as $\lim_{N\to\infty}\|\nabla^s\phi_N\|^2_{L^2(\R^n)}=1$, which follows from the properties \ref{item 1: normalization} and \ref{item 2: estimate}. 
    
    As in \cite[eq.~(20)]{RGZ2022GlobalUniqueness} we have
    \begin{equation}
    \label{eq: decomposition energy}
    \begin{split}
       E_{\gamma_i}(u_N^i)
        &= E_{\gamma_i}(u_N^i-\phi_N) + 2\langle\Theta_{\gamma_i} \nabla^s\phi_N,\nabla^s(u_N^i-\phi_N)\rangle_{L^2(\R^{2n})}+ E_{\gamma_i}(\phi_N)
    \end{split}
    \end{equation}
    for all $N\in\N$, $i=1,2$, where $u_N^i\in H^s(\R^n)$ is the unique solution to the homogeneous fractional conductivity equation with conductivity $\gamma_i$ and exterior value $\phi_N$. By Lemma~\ref{lemma: elliptic estimate} and \ref{item 2: estimate} we deduce $\|u_N-\phi_N\|_{H^s(\R^n)}\to 0$ as $N\to\infty$, which in turn implies
    \[
        \lim_{N\to\infty}E_{\gamma_i}(u_N^i)=\lim_{N\to\infty}E_{\gamma_i}(\phi_N).
    \]
    Therefore, by \eqref{eq: reconstruction} there holds
    \[
        \gamma_i(x_0)=\lim_{N\to\infty}E_{\gamma_i}(u_N^i)=\lim_{N\to\infty}B_{\gamma_i}(u_{N}^i,\phi_N)=\langle \Lambda_{\gamma_i} \phi_N,\phi_N\rangle
    \]
    for all $x_0\in V$, where in the second and third equality we used \ref{item 2 well-posedness cond eq} of Lemma~\ref{lemma: well-posedness results and DN maps}. This immediately proves the assertion of Theorem~\ref{thm: exterior determination}.
\end{proof}

\begin{remark}
    If the assumptions of Theorem~\ref{thm: exterior determination} hold for $\gamma_1$ and $\gamma_2$, then the proof shows by following the same strategy as in \cite{RGZ2022GlobalUniqueness} the reconstruction formula
    \[
        \gamma_i(x_0)=\lim_{N\to\infty}\langle \Lambda_{\gamma_i}\phi_N,\phi_N\rangle
    \]
    for a.e. $x_0\in W$ and the stability estimate
    \[
        \|\gamma_1-\gamma_2\|_{L^{\infty}(W)}\leq 2^s\|\Lambda_{\gamma_1}-\Lambda_{\gamma_2}\|_{X\to X^*}.
    \]
We note that the stability estimate could be also formulated with partial data.
\end{remark}

\subsection{Low regularity Liouville reduction}

We first prove basic estimates for products and convolutions of Bessel potential functions. These are needed to assure that the potentials associated with the studied conductivities are in the right space of Sobolev multipliers (Lemma \ref{lemma: sobolev multiplier property}). This is then used in the solution of the inverse problem by making the Liouville reduction to the fractional Schrödinger equation. 

\begin{lemma}[Multiplication estimates]
\label{lemma: multiplication estimates}
    Let $0<s<\min(1,n/2)$. Assume that $\gamma\in L^{\infty}(\R^n)$ with background deviation $m$ satisfies $\gamma(x)\geq \gamma_0>0$ and $m\in H^{s,n/s}(\R^n)$. If $u,v\in H^s(\R^n)$, then there holds $uv\in H^{s,\frac{n}{n-s}}(\R^n)$, $mu\in H^s(\R^n)$ with
        \begin{equation}
        \label{eq: uv estimate}
            \|uv\|_{H^{s,\frac{n}{n-s}}(\R^n)}\leq C\|u\|_{H^s(\R^n)}\|v\|_{H^s(\R^n)}
        \end{equation}
        and
        \begin{equation}
        \label{eq: mu estimate}
           \|mu\|_{H^s(\R^n)}\leq C(\|m\|_{L^{\infty}(\R^n)}+\|m\|_{H^{s,n/s}(\R^n)})\|u\|_{H^s(\R^n)},
        \end{equation}
where $C=C(n,s)>0$.
\end{lemma}

\begin{proof}
    We first prove the estimate \eqref{eq: uv estimate}. Let $u,v\in H^s(\R^n)$ and observe that by the Sobolev embedding we have $u,v\in L^{\frac{2n}{n-2s}}(\R^n)$. Next note that
    \[
        \frac{1}{2}+\frac{n-2s}{2n}=\frac{n-s}{n}
    \]
    and hence the exponents $p= \frac{n}{n-s},\,p_1=p_2=2,\,r_1=r_2=\frac{2n}{n-2s}$ satisfy the assumptions in Proposition~\ref{prop: Runs-Sickel Lemma I}. Now if we choose $q=2$ in Proposition~\ref{prop: Runs-Sickel Lemma I}, use the identification \ref{second identification remark A1} of Remark~\ref{remark: identification} and apply the Sobolev embedding we deduce that $uv\in H^{s,\frac{n}{n-s}}(\R^n)$ satisfying
    \[
        \begin{split}
            \|uv\|_{H^{s,\frac{n}{n-s}}(\R^n)}&\leq C(\|u\|_{H^s(\R^n)}\|v\|_{L^{\frac{2n}{n-2s}}(\R^n)}+\|v\|_{H^s(\R^n)}\|u\|_{L^{\frac{2n}{n-2s}}(\R^n)})\\
            &\leq C\|u\|_{H^s(\R^n)}\|v\|_{H^s(\R^n)}.
        \end{split}
    \]
    
    Next observe that the background deviation $m$ satisfies $m\in H^{s,n/s}(\R^n)\cap L^{\infty}(\R^n)$ and there holds
    \begin{equation}\label{eq: some-conjugate-pairs}
        \frac{s}{n}+\frac{n-2s}{2n}=\frac{1}{2}.
    \end{equation}
    Therefore the exponents $p= 2,\,p_1=n/s,\,p_2=2,\,r_1=\infty,\,r_2=\frac{2n}{n-2s}$ fulfill the conditions in Proposition~\ref{prop: Runs-Sickel Lemma I}. Now again applying Proposition~\ref{prop: Runs-Sickel Lemma I} with $q=2$, using \ref{second identification remark A1} of Remark~\ref{remark: identification} and the Sobolev embedding we deduce that $mu\in H^s(\R^n)$ satisfies
    \[
    \begin{split}
        \|mu\|_{H^s(\R^n)}&\leq C(\|m\|_{H^{s,n/s}(\R^n)}\|u\|_{L^{\frac{2n}{n-2s}}(\R^n)}+\|u\|_{H^s(\R^n)}\|m\|_{L^{\infty}(\R^n)})\\
        &\leq C(\|m\|_{H^{s,n/s}(\R^n)}+\|m\|_{L^{\infty}(\R^n)})\|u\|_{H^s(\R^n)}.
    \end{split}
    \]
    This establishes the estimate \eqref{eq: mu estimate}.
\end{proof}

\begin{lemma}[Convergence of mollifications]
\label{lemma: convergence lemma}
    Let $0<s<n/2$ and assume that $(v_k)_{k\in\N}\subset H^s(\R^n)$ satisfies $v_k\to v$ in $H^s(\R^n)$ as $k\to\infty$ for some $v\in H^s(\R^n)$. If $m\in H^{s,n/s}(\R^n)\cap L^{\infty}(\R^n)$, then there holds $m_kv_k\to mv$ in $H^s(\R^n)$ as $k\to\infty$, where $m_k\vcentcolon = \rho_{\epsilon_k}\ast m$ for some decreasing sequence $\epsilon_k\to 0$ and a sequence of standard mollifiers $(\rho_{\epsilon})_{\epsilon>0}\subset C_c^{\infty}(\R^n)$.
\end{lemma}
\begin{proof}
    By properties of mollification and $\int_{\R^n}\rho_{\epsilon}\,dx=1$, $\rho_{\epsilon}\geq 0$ for all $\epsilon>0$, we have 
    \[
        \|m_k\|_{L^{\infty}(\R^n)}\leq \|\rho_{\epsilon_k}\|_{L^1(\R^n)}\|m\|_{L^{\infty}(\R^n)}\leq \|m\|_{L^{\infty}(\R^n)}<\infty.
    \]
    By the Sobolev embedding there holds $v_k\to v$ in $L^{\frac{2n}{n-2s}}(\R^n)$ as $k\to\infty$. Since \eqref{eq: some-conjugate-pairs} holds, all conditions in assertion \ref{item 3 runst sickel II } of Proposition~\ref{prop: Runs-Sickel Lemma II} with $q=2$ are satisfied and by using statement \ref{second identification remark A1} of Remark~\ref{remark: identification} we can conclude that $m_kv_k\to mv$ in $H^s(\R^n)$ as $k\to\infty$.
\end{proof}

\begin{corollary}[Exterior conditions]
\label{corollary: regularity exterior conditions}
    Let $\Omega\subset \R^n$ be an open set and $0<s<\min(1,n/2)$. Assume that $\gamma\in L^{\infty}(\R^n)$ with background deviation $m$ satisfies $\gamma(x)\geq \gamma_0>0$ and $m\in H^{s,n/s}(\R^n)$. If $u\in\widetilde{H}^s(\Omega)$, then there holds $\gamma^{1/2}u,\,\gamma^{-1/2}u\in\widetilde{H}^s(\Omega)$.
\end{corollary}

\begin{proof}
    Let $(\rho_{\epsilon})_{\epsilon>0}\subset C_c^{\infty}(\R^n)$ be a sequence of standard mollifiers and choose a sequence $u_n\in C_c^{\infty}(\Omega)$ such that $u_n\to u$ in $H^s(\R^n)$ as $n\to\infty$. We have $m_{\epsilon}\vcentcolon = \rho_{\epsilon}\ast m\in C^{\infty}(\R^n)$ and therefore by Lemma~\ref{lemma: convergence lemma} the sequence $(m_{\epsilon_k}u_k)_{k\in\N}\subset C_c^{\infty}(\Omega)$ satisfies $m_{\epsilon_k}u_k\to mu$ in $H^s(\R^n)$ as $k\to\infty$. Hence, there holds $\gamma^{1/2}u=mu+u\in\widetilde{H}^s(\Omega)$. Next note that we can write 
    \[
       \frac{1}{\gamma^{1/2}}=1-\frac{m}{m+1}
    \]
    and set $\Gamma_0\vcentcolon =\min(0,\gamma_0^{1/2}-1)$. Let $\Gamma\in C^1_b(\R)$ satisfy $\Gamma(t)=\frac{t}{t+1}$ for $t\geq \Gamma_0$. By \cite[p.~156]{AdamsComposition} and $m\in H^{s,n/s}(\R^n)\cap L^{\infty}(\R^n)$, we deduce $\Gamma(m)\in H^{s,n/s}(\R^n)$. Since $m\geq \gamma_0^{1/2}-1$ it follows that $\frac{m}{m+1}\in H^{s,n/s}(\R^n)\cap L^{\infty}(\R^n)$. We again obtain that $\frac{m}{m+1}u\in \widetilde{H}^s(\Omega)$ and therefore $\gamma^{-1/2}u\in\widetilde{H}^s(\Omega)$.
\end{proof}

\begin{lemma}[Sobolev multiplier property]
\label{lemma: sobolev multiplier property}
    Let $0<s<\min(1,n/2)$. Assume that $\gamma\in L^{\infty}(\R^n)$ with background deviation $m$ satisfies $\gamma(x)\geq \gamma_0>0$ and $m\in H^{s,n/s}(\R^n)$. Then the distribution $q_{\gamma}=-\frac{(-\Delta)^sm}{\gamma^{1/2}}$, defined by
    \[
        \langle q_{\gamma},\phi\rangle\vcentcolon =-\langle (-\Delta)^{s/2}m,(-\Delta)^{s/2}(\gamma^{-1/2}\phi)\rangle_{L^2(\R^n)}
    \]
    for all $\phi\in C_c^{\infty}(\R^n)$, belongs to $M(H^s\rightarrow H^{-s})$. Moreover, for all $u,\phi \in H^s(\R^n)$, we have
    \begin{equation}
        \langle q_{\gamma}u,\phi\rangle = -\langle (-\Delta)^{s/2}m,(-\Delta)^{s/2}(\gamma^{-1/2}u\phi)\rangle_{L^2(\R^n)}
    \end{equation}
    satisfying the estimate
    \begin{equation}
    \label{eq: bilinear estimate}
\begin{split}
    |\langle q_{\gamma}u,\phi\rangle|&\leq C(1+\gamma_0^{-1/2}\|m\|_{L^{\infty}(\R^n)}+\|m\|_{H^{s,n/s}(\R^n)})\|m\|_{H^{s,n/s}(\R^n)}
    \\ &\quad\quad \cdot \|u\|_{H^s(\R^n)}\|\phi\|_{H^s(\R^n)}.
\end{split}
    \end{equation}
\end{lemma}

\begin{remark}
    In the rest of this article, we set $q_{\gamma}\vcentcolon =-\frac{(-\Delta)^sm_{\gamma}}{\gamma^{1/2}}\in M(H^s\rightarrow H^{-s})$ and refer to it as a potential. If no confusion can arise, we will drop the subscript $\gamma$.
\end{remark}

\begin{proof}
    Let $u,\phi\in H^s(\R^n)$ and note that by Corollary~\ref{corollary: regularity exterior conditions} we have $\gamma^{-1/2}\phi\in H^s(\R^n)$. Hence by Lemma~\ref{lemma: multiplication estimates}, it follows that $uv/\gamma^{1/2}\in H^{s,\frac{n}{n-s}}(\R^n)$ and the mapping properties of the fractional Laplacian show $(-\Delta)^{s/2}(uv/\gamma^{1/2})\in L^{\frac{n}{n-s}}(\R^n)$. Therefore we obtain by H\"older's inequality with
    \[
        \frac{s}{n}+\frac{n-s}{n}=1
    \]
    the estimate
    \begin{equation}
    \label{eq: estimate}
    \begin{split}
        &|\langle (-\Delta)^{s/2}m,(-\Delta)^{s/2}(\gamma^{-1/2}u\phi)\rangle_{L^2(\R^n)}|\\
        &\leq \|(-\Delta)^{s/2}m\|_{L^{n/s}(\R^n)}\|(-\Delta)^{s/2}(\gamma^{-1/2}u\phi)\|_{L^{\frac{n}{n-s}}(\R^n)}\\
        &\leq \|m\|_{H^{s,n/s}(\R^n)}\|\gamma^{-1/2}u\phi\|_{H^{s,\frac{n}{n-s}}(\R^n)}\\
        &\leq \|m\|_{H^{s,n/s}(\R^n)}(\|u\phi\|_{H^{s,\frac{n}{n-s}}(\R^n)}+\|\tfrac{m}{m+1}u\phi\|_{H^{s,\frac{n}{n-s}}(\R^n)}).
    \end{split}
    \end{equation}
    In the last inequality we used the decompositions $\gamma^{-1/2}=1-\frac{m}{m+1}$ with $\frac{m}{m+1}\in H^{s,n/s}(\R^n)\cap L^{\infty}(\R^n)$ (see the proof of Corollary~\ref{corollary: regularity exterior conditions}). 
    
    Using \eqref{eq: uv estimate} we can estimate the expression in brackets in the equation \eqref{eq: estimate}, \eqref{eq: mu estimate} as
    \[
    \begin{split}
        &\|u\phi\|_{H^{s,\frac{n}{n-s}}(\R^n)}+\|\tfrac{m}{m+1}u\phi\|_{H^{s,\frac{n}{n-s}}(\R^n)}\\
        &\leq C(\|u\|_{H^s(\R^n)}\|\phi\|_{H^s(\R^n)}+\|\tfrac{m}{m+1}u\|_{H^s(\R^n)}\|\phi\|_{H^s(\R^n)})\\
        &\leq C(1+\|\tfrac{m}{m+1}\|_{L^{\infty}(\R^n)}+\|\tfrac{m}{m+1}\|_{H^{s,n/s}(\R^n)})\|u\|_{H^s(\R^n)}\|\phi\|_{H^s(\R^n)}.
    \end{split}
    \]
    Next observe that by \cite[p.~156]{AdamsComposition} there holds 
    \[
        \left\|\frac{m}{m+1}\right\|_{H^{s,n/s}(\R^n)}\leq C\|m\|_{H^{s,n/s}(\R^n)}
    \]
    for some $C>0$ and $0<s<1$. This implies 
    \begin{equation}
    \label{eq: estimate bilinear q}
    \begin{split}
        &|\langle (-\Delta)^{s/2}m,(-\Delta)^{s/2}(\gamma^{-1/2}u\phi)\rangle_{L^2(\R^n)}|\\
        &\leq C(1+\gamma_0^{-1/2}\|m\|_{L^{\infty}(\R^n)}+\|m\|_{H^{s,n/s}(\R^n)})\|m\|_{H^{s,n/s}(\R^n)}\\
        &\quad\quad \cdot\|u\|_{H^s(\R^n)}\|\phi\|_{H^s(\R^n)}.
    \end{split}
    \end{equation}
    Now by using cutoff functions and the estimate \eqref{eq: estimate bilinear q}, one may notice that $q_\gamma$ defines a distribution. Furthermore, then it directly follows that $q_{\gamma}\in M(H^s\rightarrow H^{-s})$. Finally, by the equations \eqref{eq: uv estimate}, \eqref{eq: mu estimate} and H\"older's inequality one can deduce that there holds
    \[
        \langle q_{\gamma}u,\phi\rangle =\langle (-\Delta)^{s/2}m,(-\Delta)^{s/2}(\gamma^{-1/2}u\phi)\rangle_{L^2(\R^n)}
    \]
    for all $u,\phi\in H^s(\R^n)$ (see \eqref{eq: multiplication by sobolev multiplier}). Therefore, we can conclude the proof.
\end{proof} 

\begin{lemma}[Liouville reduction]
\label{lemma: Liouville reduction}
    Let $0<s<\min(1,n/2)$. Assume that $\gamma\in L^{\infty}(\R^n)$ with conductivity matrix $\Theta$ and background deviation $m$ satisfies $\gamma(x)\geq \gamma_0>0$ and $m\in H^{s,n/s}(\R^n)$. Then there holds
    \begin{equation}
    \label{eq: Liouville reduction}
    \begin{split}
        \langle\Theta\nabla^su,\nabla^s\phi\rangle_{L^2(\R^{2n})}=&\,\langle (-\Delta)^{s/2}(\gamma^{1/2}u),(-\Delta)^{s/2}(\gamma^{1/2}\phi))\rangle_{L^2(\R^n)}\\
        &+\langle q(\gamma^{1/2}u),(\gamma^{1/2}\phi)\rangle
    \end{split}
    \end{equation}
    for all $u,\phi\in H^s(\R^n)$.
\end{lemma}

\begin{proof} We first prove identity \eqref{eq: Liouville reduction} for Schwartz functions and then extend it to functions in $H^s(\R^n)$ by an approximation argument.\\
    \textbf{Step 1}: Let $u,\phi\in\schwartz(\R^n)$ and define $\gamma^{1/2}_{\epsilon}\vcentcolon = \gamma^{1/2}\ast \rho_{\epsilon}\in C^{\infty}_b(\R^n)$, $m_{\epsilon}\vcentcolon = m\ast \rho_{\epsilon}\in C^{\infty}_b(\R^n)\cap H^{s,n/s}(\R^n)$. Here $(\rho_{\epsilon})_{\epsilon>0}\subset C_c^{\infty}(\R^n)$ is a sequence of standard mollifiers and $C^{\infty}_b(\R^n)$ denotes the space of smooth functions with bounded derivatives. We have $m_{\epsilon}=\gamma^{1/2}_{\epsilon}-1$. From \cite[Proof of Theorem 8.6]{RZ2022unboundedFracCald}, we know 
\[
    \lim_{\epsilon\to 0}\langle \Theta_{\epsilon}\nabla^su,\nabla^s\phi\rangle_{L^2(\R^{2n})}=\langle \Theta\nabla^su,\nabla^s\phi\rangle_{L^2(\R^{2n})}
\]
for all $u,\phi\in H^s(\R^n)$, where  $\Theta_{\epsilon}=\gamma^{1/2}_{\epsilon}(x)\gamma^{1/2}_{\epsilon}(y)\mathbf{1}_{n\times n}$. Using \cite[Remark 8.9]{RZ2022unboundedFracCald}, we deduce that $m_{\epsilon}\in C^t(\R^n)\cap L_s(\R^n)$ for all $t\in \R_+\setminus \N$ and by \cite[Proposition 2.1.4]{SilvestreFracObstaclePhd} there holds
\[
    (-\Delta)^{s}m_{\epsilon}(x)=-\frac{C_{n,s}}{2}\int_{\R^n}\frac{\delta m_{\epsilon}(x,y)}{|y|^{n+2s}}\,dy
\]
for all $x\in\R^n$. Here $C^t(\R^n)$, $t\in\R_+\setminus\N$, denotes the space of all H\"older continuous functions and $L_s(\R^n)$ stands for the space of all functions $f\in L^1_{loc}(\R^n)$ such that 
\[
    \int_{\R^n}\frac{|f(x)|}{1+|x|^{n+2s}}\,dx<\infty.
\]
Thus, the proof of \cite[Theorem 8.6]{RZ2022unboundedFracCald} shows
\begin{equation}
\label{eq: approximation term}
\begin{split}
    \langle \Theta\nabla^su,\nabla^s\phi\rangle_{L^2(\R^{2n})}=&\lim_{\epsilon\to 0} \big(\langle (-\Delta)^{s/2}(\gamma^{1/2}_{\epsilon}u), (-\Delta)^{s/2}(\gamma^{1/2}_{\epsilon}\phi)\rangle_{L^2(\R^n)}\\
    &-\langle (-\Delta)^{s}m_{\epsilon},\gamma^{1/2}_{\epsilon}u\phi \rangle_{L^2(\R^n)}\big)
\end{split}
\end{equation}
for all $u,\phi\in\schwartz(\R^n)$. 

Since $u,\phi\in\schwartz(\R^n)$ and $\gamma^{1/2}_{\epsilon}\in C^{\infty}_b(\R^n)$, we have $\gamma^{1/2}_{\epsilon}u\phi\in\schwartz(\R^n)$. Therefore, by the fact that $m_{\epsilon}\in H^{t,n/s}(\R^n)$ for all $t\in\R$, we have by approximation
\[
    \langle (-\Delta)^{s}m_{\epsilon},\gamma^{1/2}_{\epsilon}u\phi \rangle_{L^2(\R^n)}=\langle (-\Delta)^{s/2}m_{\epsilon},(-\Delta)^{s/2}(\gamma^{1/2}_{\epsilon}u\phi)\rangle_{L^2(\R^n)}.
\]
Using $m_{\epsilon}=\gamma_{\epsilon}^{1/2}-1$, we can decompose the last expression as
\[
\begin{split}
    \langle (-\Delta)^{s}m_{\epsilon},\gamma^{1/2}_{\epsilon}u\phi \rangle_{L^2(\R^n)}=\,&\langle (-\Delta)^{s/2}m_{\epsilon},(-\Delta)^{s/2}(m_{\epsilon}u\phi) \rangle_{L^2(\R^n)}\\
    &+\langle (-\Delta)^{s/2}m_{\epsilon},(-\Delta)^{s/2}(u\phi) \rangle_{L^2(\R^n)}
\end{split}
\]
for all $u,\phi\in\schwartz(\R^n)$. By our regularity assumptions, we know $(-\Delta)^{s/2}m_{\epsilon}\to (-\Delta)^{s/2}m$ in $L^{n/s}(\R^n)$. Thus, by the mapping properties of the fractional Laplacian and H\"older's inequality, we have
\[
    \langle (-\Delta)^{s/2}m_{\epsilon},(-\Delta)^{s/2}(u\phi) \rangle_{L^2(\R^n)}\to \langle (-\Delta)^{s/2}m,(-\Delta)^{s/2}(u\phi) \rangle_{L^2(\R^n)}
\]
as $\epsilon\to 0$. By Lemma~\ref{lemma: multiplication estimates}, we can estimate
\begin{equation}
\begin{split}
    \|m_{\epsilon}u\phi-mu\phi\|_{H^{s,\frac{n}{n-s}}(\R^n)}\leq & C\|m_{\epsilon}u-mu\|_{H^s(\R^n)}\|\phi\|_{H^s(\R^n)}
\end{split}
\end{equation}
and since Lemma~\ref{lemma: convergence lemma} implies $m_{\epsilon}u\to mu$ in $H^s(\R^n)$ as $\epsilon\to 0$ we deduce $m_{\epsilon}u\phi\to mu\phi$ in $H^{s,\frac{n}{n-s}}(\R^n)$. By the mapping properties of the fractional Laplacian, we have
\[
    (-\Delta)^{s/2}(m_{\epsilon}u\phi)\to (-\Delta)^{s/2}(mu\phi)\quad\text{in}\quad L^{\frac{n}{n-s}}(\R^n)
\]
as $\epsilon\to 0$ and therefore H\"older's inequality shows 
\[
    \langle (-\Delta)^{s/2}m_{\epsilon},(-\Delta)^{s/2}(m_{\epsilon}u\phi) \rangle_{L^2(\R^n)}\to \langle (-\Delta)^{s/2}m,(-\Delta)^{s/2}(mu\phi) \rangle_{L^2(\R^n)}
\]
as $\epsilon\to 0$. Thus, we have shown that there holds
\[
\begin{split}
    \lim_{\epsilon\to 0}\langle (-\Delta)^{s}m_{\epsilon},\gamma^{1/2}_{\epsilon}u\phi \rangle_{L^2(\R^n)}=&\,\langle (-\Delta)^{s/2}m,(-\Delta)^{s/2}(mu\phi)\rangle\\
    &+ \langle (-\Delta)^{s/2}m,(-\Delta)^{s/2}(u\phi)\rangle\\
    =&\,\langle (-\Delta)^{s/2}m, (-\Delta)^{s/2}(\gamma^{1/2}u\phi)\rangle
\end{split}
\]
for all $u,\phi\in \schwartz(\R^n)$. 

By the usual splitting $\gamma_{\epsilon}^{1/2}=m_{\epsilon}+1$, we see that the first term in \eqref{eq: approximation term} can be decomposed as
\begin{equation}
\label{eq: splitting}
\begin{split}
    &\langle (-\Delta)^{s/2}(\gamma^{1/2}_{\epsilon}u), (-\Delta)^{s/2}(\gamma^{1/2}_{\epsilon}\phi)\rangle_{L^2(\R^n)}\\
    &=\,\langle (-\Delta)^{s/2}(m_{\epsilon}u), (-\Delta)^{s/2}(m_{\epsilon}\phi)\rangle_{L^2(\R^n)}\\
    &\quad+\langle (-\Delta)^{s/2}(m_{\epsilon}u), (-\Delta)^{s/2}\phi\rangle_{L^2(\R^n)}\\
    &\quad+\langle (-\Delta)^{s/2}u, (-\Delta)^{s/2}(m_{\epsilon}\phi)\rangle_{L^2(\R^n)}\\
    &\quad+\langle (-\Delta)^{s/2}u, (-\Delta)^{s/2}\phi\rangle_{L^2(\R^n)}.
\end{split}
\end{equation}
    Now Lemma~\ref{lemma: convergence lemma} and the continuity of the fractional Laplacian imply 
    \[
        (-\Delta)^{s/2}(m_{\epsilon}v)\to (-\Delta)^{s/2}(mv)\quad\text{in}\quad L^2(\R^n)
    \]
    as $\epsilon\to 0$ for all $v\in H^s(\R^n)$. This shows
    \[\begin{split}
        &\langle (-\Delta)^{s/2}(\gamma^{1/2}_{\epsilon}u), (-\Delta)^{s/2}(\gamma^{1/2}_{\epsilon}\phi)\rangle_{L^2(\R^n)}\\
        &\to \langle (-\Delta)^{s/2}(\gamma^{1/2}u), (-\Delta)^{s/2}(\gamma^{1/2}\phi)\rangle_{L^2(\R^n)}
        \end{split}
    \]
    for all $u,\phi\in\schwartz(\R^n)$ and hence by the definition of $q$ the identity \eqref{eq: Liouville reduction} holds for all $u,\phi\in \schwartz(\R^n)$.\\
    
    \noindent\textbf{Step 2}: Let $(u_k)_{k\in\N},(\phi_k)_{k\in\N}\subset C_c^{\infty}(\R^n)$ such that $u_k\to u$ and $\phi_k\to \phi$ in $H^s(\R^n)$ as $k\to\infty$. By the first step, we have
    \begin{equation}
    \label{eq: convergence for Hs}
    \begin{split}
    \langle\Theta\nabla^su_k,\nabla^s\phi_k\rangle_{L^2(\R^{2n})}=&\,\langle (-\Delta)^{s/2}(\gamma^{1/2}u_k),(-\Delta)^{s/2}(\gamma^{1/2}\phi_k))\rangle_{L^2(\R^n)}\\
    &-\langle q(\gamma^{1/2}u_k),(\gamma^{1/2}\phi_k)\rangle
    \end{split}
    \end{equation}
    for all $k\in\N$. By continuity of the fractional gradient from $H^s(\R^n)$ to $L^2(\R^{2n})$ and $\phi_k\to \phi$, $u_k\to u$ in $H^s(\R^n)$ as $k\to\infty$, we have
    \[
        \langle\Theta\nabla^su_k,\nabla^s\phi_k\rangle_{L^2(\R^{2n})}\to \langle\Theta\nabla^su,\nabla^s\phi\rangle_{L^2(\R^{2n})}
    \]
    as $k\to\infty$. Using again the splitting \eqref{eq: splitting} and Lemma~\ref{lemma: convergence lemma}, we deduce
    \[\begin{split}
        &\langle (-\Delta)^{s/2}(\gamma^{1/2}u_k), (-\Delta)^{s/2}(\gamma^{1/2}\phi_k)\rangle_{L^2(\R^n)}
        \\&\to \langle (-\Delta)^{s/2}(\gamma^{1/2}u), (-\Delta)^{s/2}(\gamma^{1/2}\phi)\rangle_{L^2(\R^n)}
        \end{split}
    \]
    as $k\to\infty$. Finally, the convergence of the second term in \eqref{eq: convergence for Hs} follows from the estimate \eqref{eq: bilinear estimate} in Lemma~\ref{lemma: sobolev multiplier property} and Lemma~\ref{lemma: convergence lemma}. Hence, we can conclude the proof. 
\end{proof}

\begin{definition}[Bilinear form and weak solutions to the Schr\"odinger equation]
\label{def: weak solutions Schrödinger eq}
    Let $\Omega\subset\R^n$ be an open set with nonempty exterior, $s>0$ and $q\in M(H^s\rightarrow H^{-s})$. Then we define the (continuous) bilinear form related to the fractional Schr\"odinger equation with potential $q$ by $B_q\colon H^s(\R^n)\times H^s(\R^n)\to \R$, where
    \[
       B_q(u,v)\vcentcolon =\int_{\R^n}(-\Delta)^{s/2}u\,(-\Delta)^{s/2}v\,dx+\langle qu,v\rangle
     \]
     for all $u,v\in H^s(\R^n)$. Moreover, if $f\in H^s(\R^n)$ and $F\in (\widetilde{H}^s(\Omega))^*$, then we say that $v\in H^s(\R^n)$ is a weak solution to the fractional Schr\"odinger equation
    \begin{equation}
            \begin{split}
            ((-\Delta)^s+q)v&=0\quad\text{in}\quad\Omega,\\
            v&=f\quad\text{in}\quad\Omega_e
        \end{split}
    \end{equation}
     if there holds $B_q(v,\phi)=F(\phi)$ for all $\phi\in \widetilde{H}^s(\Omega)$ and $v-f\in\widetilde{H}^s(\Omega)$.
\end{definition}

\begin{lemma}[Well-posedness and DN maps for the Schr\"odinger equation]
\label{eq: well-posedness results and DN maps}
    Let $\Omega\subset \R^n$ be an open set which is bounded in one direction and $0<s<\min(1,n/2)$. Assume that $\gamma\in L^{\infty}(\R^n)$ with conductivity matrix $\Theta$, background deviation $m$ and potential $q$ satisfies $\gamma(x)\geq \gamma_0>0$ and $m\in H^{s,n/s}(\R^n)$. Then the following assertions hold:
    \begin{enumerate}[(i)]
        \item\label{item 1 well-posedness cond eq and DN maps} If $u\in H^s(\R^n)$, $g\in X\vcentcolon = H^s(\R^n)/\widetilde{H}^s(\Omega)$ and $v\vcentcolon =\gamma^{1/2}u,\,f\vcentcolon =\gamma^{1/2}g$. Then $v\in H^s(\R^n), f\in X$ and $u$ is a weak solution of the fractional conductivity equation
        \begin{equation}
        \label{eq: fractional conductivity equation}
        \begin{split}
            \Div_s(\Theta\nabla^s u)&= 0\quad\text{in}\quad\Omega,\\
            u&= g\quad\text{in}\quad\Omega_e
        \end{split}
        \end{equation}
        if and only if $v$ is a weak solution of the fractional Schr\"odinger equation
        \begin{equation}
        \label{eq: fractional Schroedinger equation}    
            \begin{split}
            ((-\Delta)^s+q_{\gamma})v&=0\quad\text{in}\quad\Omega,\\
            v&=f\quad\text{in}\quad\Omega_e.
        \end{split}
        \end{equation}
        \item\label{item 2 well-posedness cond eq and DN maps} Conversely, if $v\in H^s(\R^n), f\in X$ and $u\vcentcolon =\gamma^{-1/2}v,\,g\vcentcolon =\gamma^{-1/2}f$. Then $v$ is a weak solution of \eqref{eq: fractional Schroedinger equation} if and only if $u$ is a weak solution of \eqref{eq: fractional conductivity equation}.
        \item\label{item 3 well-posedness cond eq and DN maps} For all $f\in X$ there is a unique weak solutions $v_f\in H^s(\R^n)$ of the fractional Schr\"odinger equation
        \begin{equation}
            \begin{split}
            ((-\Delta)^s+q)v&=0\quad\text{in}\quad\Omega,\\
            v&=f\quad\text{in}\quad\Omega_e.
        \end{split}
        \end{equation}
        \item\label{item 4 well-posedness cond eq and DN maps} The exterior DN map $\Lambda_q\colon X\to X^*$ given by 
        \[
        \begin{split}
            \langle \Lambda_qf,g\rangle\vcentcolon =B_q(v_f,g),
        \end{split}
        \]
        where $v_f\in H^s(\R^n)$ is the unique solution to the Schr\"odinger equation with exterior value $f$, is a well-defined bounded linear map. 
    \end{enumerate}
\end{lemma}

\begin{proof} Since the proofs for the assertion \ref{item 1 well-posedness cond eq and DN maps} and \ref{item 2 well-posedness cond eq and DN maps} are very similar we only show \ref{item 1 well-posedness cond eq and DN maps}.

\ref{item 1 well-posedness cond eq and DN maps} First of all note that by Lemma~\ref{lemma: sobolev multiplier property} the notion of weak solutions to the fractional Schr\"odinger equation is well-defined. Moreover, from Corollary~\ref{corollary: regularity exterior conditions} and the same argument as in \cite[Theorem 8.6]{RZ2022unboundedFracCald} it follows that $v\in H^s(\R^n)$ and $f\in X$. Now the rest of the assertion follows from the Liouville reduction (Lemma~\ref{lemma: Liouville reduction}).

\ref{item 3 well-posedness cond eq and DN maps} Let $g\vcentcolon = \gamma^{-1/2}f\in X$ then by Lemma~\ref{lemma: well-posedness results and DN maps} there is a unique solution $u_g\in H^s(\R^n)$ to the fractional conductivity equation with exterior value $g$. Using assertion \ref{item 1 well-posedness cond eq and DN maps} we deduce that $v_f\vcentcolon = \gamma^{1/2}u_g\in H^s(\R^n)$ solves \eqref{eq: fractional Schroedinger equation} with exterior value  $f$. The obtained solution is unique since solutions to the fractional conductivity equation are.

\ref{item 4 well-posedness cond eq and DN maps} This last assertion follows from the continuity of the fractional Laplacian, Lemma~\ref{lemma: sobolev multiplier property} and the fact that the solutions to the Schr\"odinger equation depend continuously on the data.
\end{proof}

\begin{remark} The equations \eqref{eq: fractional conductivity equation} and \eqref{eq: fractional Schroedinger equation} also have the Runge approximation property. This follows from the abstract theory in \cite[Section 4]{RZ2022unboundedFracCald}, and in particular from \cite[Remark 4.2 and Theorem 4.3]{RZ2022unboundedFracCald}. The original argument is based on the work \cite{GSU20}. We give two proofs of Theorem~\ref{thm: Global uniqueness} in this article. The first one is based on the Runge approximation property and standard arguments. The second proof, based on a new strategy, does not rely on the Runge approximation property.
\end{remark}

\subsection{First proof of Theorem \ref{thm: Global uniqueness}}
We present the first proof of Theorem \ref{thm: Global uniqueness} in this section. It follows the standard structure applied earlier in different regularity settings in \cite{RGZ2022GlobalUniqueness,C21,RZ2022unboundedFracCald}.

First, we state a useful basic lemma as a preparation. The proof of Lemma \ref{lemma: Relation DN maps} is virtually identical to that of \cite[Lemma 4.1]{RGZ2022GlobalUniqueness} and follows from Corollary~\ref{corollary: regularity exterior conditions}, and Lemmas~\ref{lemma: Liouville reduction} and \ref{eq: well-posedness results and DN maps} considering the low regularity setting. Therefore, we do not repeate these details here.

\begin{lemma}
\label{lemma: Relation DN maps}
    Let $\Omega\subset\R^n$ be an open set which is bounded in one direction, $W\subset\Omega_e$ an open set and $0<s<\min(1,n/2)$. Assume that $\gamma,\Gamma\in L^{\infty}(\R^n)$ with background deviations $m_{\gamma},m_{\Gamma}$ satisfy $\gamma(x),\Gamma(x)\geq \gamma_0>0$ and $m_{\gamma},m_{\Gamma}\in H^{s,n/s}(\R^n)$. If $\gamma|_{W}=\Gamma|_{W}$, then
    \begin{equation}
    \label{eq: identity DN maps}
        \langle \Lambda_{\gamma}f,g\rangle=\langle \Lambda_{q_\gamma}(\Gamma^{1/2}f),(\Gamma^{1/2}g)\rangle
    \end{equation}
    holds for all $f,g\in\widetilde{H}^{s}(W)$.
\end{lemma}

We can now give the first proof of Theorem~\ref{thm: Global uniqueness}.

\begin{proof}[First proof of Theorem \ref{thm: Global uniqueness}] We have that $\gamma_1|_W=\gamma_2|_W$ a.e. by exterior determination (Theorem \ref{thm: exterior determination}). Throughout the proof we choose any function $\Gamma\in L^{\infty}(\R^n)$ satisfying the conditions $\Gamma\geq \gamma_0$, $m_{\Gamma}\vcentcolon =\Gamma^{1/2}-1\in H^{s,n/s}(\R^n)$ and $\Gamma=\gamma_1=\gamma_2$ in $W$. 

By the argument in \cite[Proof of Theorem 1.1]{RGZ2022GlobalUniqueness} we know that there holds
    \[
        \langle \Lambda_{\gamma_1}f,g\rangle=\langle \Lambda_{\gamma_2}f,g\rangle
    \]
    for all $f,g\in \widetilde{H}^s(W)$ if and only if we have
    \[
        \langle \Lambda_{\gamma_1}f,g\rangle=\langle \Lambda_{\gamma_2}f,g\rangle
    \]
    for all $f,g\in C_c^{\infty}(W)$. Similarly, one can show by Lemma~\ref{lemma: convergence lemma} and \ref{lemma: sobolev multiplier property} that there holds
     \[
    \langle \Lambda_{q_1}f,g\rangle=\langle \Lambda_{q_2}f,g\rangle
    \]
    for all $f,g\in C_c^{\infty}(W)$ if and only if
    \[
    \langle \Lambda_{q_1}f,g\rangle=\langle \Lambda_{q_2}f,g\rangle
    \]
for all $f,g\in \widetilde{H}^s(W)$.

Now using Lemma~\ref{lemma: Relation DN maps}, we deduce that the condition $\Lambda_{\gamma_1}f|_{W}=\Lambda_{\gamma_2}f|_{W}$ for all $f\in C_c^{\infty}(W)$ is equivalent to 
\begin{equation}
\label{eq: equivalence 1}
     \langle \Lambda_{q_1}(\Gamma^{1/2}f),(\Gamma^{1/2}g)\rangle=\langle \Lambda_{q_2}(\Gamma^{1/2}f),(\Gamma^{1/2}g)\rangle
\end{equation}
for all $f,g\in \widetilde{H}^s(W)$. By Corollary~\ref{corollary: regularity exterior conditions}, we can replace $f,g$ in this identity by $\Gamma^{-1/2}f,\Gamma^{-1/2}g\in\widetilde{H}^s(W)$ and obtain \eqref{eq: equivalence 1} is equivalent to
\begin{equation}
\label{eq: equivalence 2}
     \langle \Lambda_{q_1}f,g\rangle=\langle \Lambda_{q_2}f,g\rangle  
\end{equation}
for all $f,g\in \widetilde{H}^s(W)$.

Using \eqref{eq: equivalence 2} and the uniqueness results for the fractional Schrödinger equations \cite[Theorem 2.2, Corollary 2.7 and Remark 4.2]{RZ2022unboundedFracCald}, we obtain $q_1=q_2$ in the weak sense in $\widetilde{H}^s(W \cup \Omega)$. This implies that
\[
    \langle (-\Delta)^{s/2}(m_1-m_2),(-\Delta)^{s/2}(fg)\rangle =0
\]
for all $f,g\in\widetilde{H}^s(W)$, where we have again used Corollary~\ref{corollary: regularity exterior conditions} to replaced $g\in \widetilde{H}^s(W)$ by $\Gamma^{1/2}g$ and $\gamma_1=\gamma_2=\Gamma$ in $W$. Now choosing $f\in C_c^{\infty}(\omega)$ and a cutoff function $g\in C_c^{\infty}(W)$ with $g|_{\omega}=1$, where $\omega\Subset W$ is some nonempty open set, we see that $(-\Delta)^{s}(m_1-m_2)=0$ in $\omega$. 
On the other hand, we have by the assumption $\gamma_1=\gamma_2$ in $\omega$ and hence the UCP (Theorem~\ref{thm:UCP-general}) implies $m_1\equiv m_2$. This shows $\gamma_1\equiv\gamma_2$, which concludes the proof.
\end{proof}

\subsection{Second proof of Theorem \ref{thm: Global uniqueness}}
\label{sec: second proof of main theorem}

In this subsection, we give an alternative proof of Theorem \ref{thm: Global uniqueness} which is not appearing in the earlier literature. This argument is based on using energies of the solutions in combination with the UCP. The merit of this argument is that it avoids using the Runge approximation property and a full scale reduction to the fractional Schrödinger equation. The approach resembles the proof of single measurement uniqueness for the Calderón problem of the fractional Schrödinger equation \cite{GRSU-fractional-calderon-single-measurement}. Heuristically speaking, this proof is based more directly on the properties of the fractional conductivity equation than the other argument, which reduces the problem to the data $\Lambda_{q_1}=\Lambda_{q_2}$.

Let $s\geq 0$ and $\gamma_0>0$. We write that $\gamma \in M_{\gamma_0}^s$ if the following holds:
\begin{enumerate}[(i)]
    \item  $\gamma\in L^{\infty}(\R^n)$ with $\gamma \geq \gamma_0 > 0$ a.e. in $\R^n$;
    \item $\gamma^{1/2} u \in \widetilde{H}^s(\Omega)$ and $\gamma^{-1/2} u \in \widetilde{H}^s(\Omega)$ for any open set $\Omega \subset \R^n$ and $u \in \widetilde{H}^s(\Omega)$.
\end{enumerate}

\begin{lemma}[Relation of solutions]\label{lemma: relation for solutions} Let $\Omega\subset \R^n$ be an open set which is bounded in one direction and $0<s<1$. Assume that $\gamma_1,\gamma_2\in M_{\gamma_0}^s$ satisfy $\gamma_1|_{W_2} = \gamma_2|_{W_2}$ and let $W_1,W_2 \subset \Omega_e$ be nonempty open sets.
If $\left.\Lambda_{\gamma_1}f\right|_{W_2}=\left.\Lambda_{\gamma_2}f\right|_{W_2}$ for some $f \in \widetilde{H}^s(W_1)$ with $W_2\setminus\supp(f)\neq \emptyset$, then there holds
\begin{equation}\label{eq: relation for solutions}
\gamma_1^{1/2}u_f^1 = \gamma_2^{1/2}u_f^2\quad \text{a.e. in $\R^n$}.
\end{equation}
\end{lemma}
\begin{proof} Choose first some $f,\phi \in H^s(\R^n)$ with disjoint supports. For any $\gamma \in M_{\gamma_0}^s$ there holds $\gamma^{1/2}f,\gamma^{1/2}\phi \in H^s(\R^n)$, and hence we obtain
\begin{equation}\label{eq:DisjointlowregLiouville}
    B_\gamma(f,\phi) = \langle (-\Delta)^{s/2}(\gamma^{1/2}f),(-\Delta)^{s/2}(\gamma^{1/2}\phi)\rangle_{L^2(\R^n)}
\end{equation}
by expanding both terms into their quadratic forms and applying the disjoint support condition (cf.~\cite[eq. (3) and Theorem 2]{TruncationSobolev}). 

Let now $f \in \widetilde{H}^s(W_1)$ be as in the statement and denote by $u_f^1$, $u_f^2$ the corresponding unique solutions to the homogeneous fractional conductivity equations with the conductivities $\gamma_1, \gamma_2$ and exterior value $f$. By the assumption $V\vcentcolon = W_2\setminus\supp(f)$ is a nonempty open set. Then for all $\phi \in \widetilde{H}^s(V)$ we have by the definition of the DN maps and $\Lambda_{\gamma_1}f|_{W_2}=\Lambda_{\gamma_2}f|_{W_2}$ the identity
\begin{equation}\label{eq: DNmapsEquality}
    B_{\gamma_1}(u_f^1,\phi)=\ip{\Lambda_{\gamma_1}f}{\phi}=\ip{\Lambda_{\gamma_2}f}{\phi}=B_{\gamma_2}(u_f^2,\phi).
\end{equation}
Using \eqref{eq:DisjointlowregLiouville} and $\supp(f)\cap \supp(\phi)=\emptyset$, we obtain
\begin{equation}
\begin{split}
     &\langle (-\Delta)^{s/2}(\gamma_1^{1/2}u_f^1),(-\Delta)^{s/2}(\gamma_1^{1/2}\phi)\rangle_{L^2(\R^n)}\\
     &= \langle(-\Delta)^{s/2}(\gamma_2^{1/2}u_f^2),(-\Delta)^{s/2}(\gamma_2^{1/2}\phi)\rangle_{L^2(\R^n)}.
\end{split}
\end{equation}
Since $\gamma_1^{1/2}|_{W_2}=\gamma_2^{1/2}|_{W_2}$, we have by approximation $\gamma_1^{1/2}\phi=\gamma_2^{1/2}\phi$ in $\R^n$ and obtain
\begin{equation}
     \langle(-\Delta)^{s/2}(\gamma_1^{1/2}u_f^1-\gamma_2^{1/2}u_f^2),(-\Delta)^{s/2}(\gamma_1^{1/2}\phi)\rangle_{L^2(\R^n)}=0.
\end{equation}
By choosing $\phi = \gamma_{1}^{-1/2}g \in \widetilde{H}^s(V)$, we get
\begin{equation}
     \langle(-\Delta)^{s/2}(\gamma_1^{1/2}u_f^1-\gamma_2^{1/2}u_f^2),(-\Delta)^{s/2}g\rangle_{L^2(\R^n)}=0.
\end{equation}
for any $g \in \widetilde{H}^s(V)$. Since $\gamma_1^{1/2}u_f^1-\gamma_2^{1/2}u_f^2 \in H^s(\R^n)$ and $\gamma_1^{1/2}u_f^1-\gamma_2^{1/2}u_f^2=0$ in $V$, the UCP of the fractional Laplacian gives $\gamma_1^{1/2}u_f^1=\gamma_2^{1/2}u_f^2$ in $\R^n$.
\end{proof}

\begin{remark} Suppose that $\Omega, \gamma_1,\gamma_2$ satisfy the assumptions of Lemma \ref{lemma: relation for solutions} and $W_1,W_2 \subset \Omega_e$ are nonempty open sets. Suppose that the following holds: For all $V \Subset W_1$ it holds that $W_2 \setminus V \neq \emptyset$. Then $\Lambda_{\gamma_1}f|_{W_2} = \Lambda_{\gamma_2}f|_{W_2}$ for all $f \in C_c^\infty(W_1)$ implies that $\gamma_1^{1/2}u_g^1 = \gamma_2^{1/2}u_g^2$ a.e. in $\R^n$ for any $g \in \widetilde{H}^s(W_1)$. This follows by approximation and Lemma \ref{lemma: relation for solutions}:
We have by the linearity of the fractional conductivity operator and the Lax--Milgram theorem that the solutions depend continuously on the exterior conditions in $H^s(\R^n)$. Therefore, the general case $g \in \widetilde{H}^s(W_1)$ follows by taking a sequence $(f_j)_{j \in \N} \subset C_c^\infty(W_1)$ such that $f_j \to g$ in $H^s(\R^n)$. Now there holds $\gamma_1^{1/2}u_{f_j}^1=\gamma_2^{1/2}u_{f_j}^2$ a.e. in $\R^n$ for all $j \in \N$ by Lemma \ref{lemma: relation for solutions}. Hence, the identity $\gamma_1^{1/2}u_{g}^1=\gamma_2^{1/2}u_{g}^2$ also holds a.e. in $\R^n$ by taking the limit $j \to \infty$ of a suitable subsequence. One can for example take $W := W_1 =W_2$.
\end{remark}

\begin{proof}[Second proof of Theorem \ref{thm: Global uniqueness}]
We have that $\gamma_1|_W=\gamma_2|_W$ a.e. by exterior determination (Theorem \ref{thm: exterior determination}). By the Alessandrini identity \cite[Lemma 4.4]{RZ2022unboundedFracCald} and the symmetry of the bilinear forms, there holds
\begin{equation}\label{eq: Alessandrini}
    0=\langle(\Lambda_{\gamma_1}-\Lambda_{\gamma_2})f,g\rangle = (B_{\gamma_1}-B_{\gamma_2})(u_f^1,u_g^2)
\end{equation}
for any $f,g \in X$. Let now $f=g \in \widetilde{H}^s(V)$ for some $V\Subset W$. By the Liouville reduction (Lemma \ref{lemma: Liouville reduction}), \eqref{eq: Alessandrini}, and $\gamma_1|_W = \gamma_2|_W$, we obtain
\[
    \begin{split}
        0=\,&\langle(\Lambda_{\gamma_1}-\Lambda_{\gamma_2})f,g\rangle=(B_{\gamma_1}-B_{\gamma_2})(u_f^1,u_f^2)\\
        =\,& \langle (-\Delta)^{s/2}(\gamma_1^{1/2}u_f^1),(-\Delta)^{s/2}(\gamma_1^{1/2}f)\rangle_{L^2(\R^n)}\\
        &-\langle (-\Delta)^{s/2}m_1,(-\Delta)^{s/2}(\gamma_1^{1/2}f^2 )\rangle_{L^2(\R^n)}\\
        &-\langle (-\Delta)^{s/2}(\gamma_2^{1/2}u_f^2),(-\Delta)^{s/2}(\gamma_2^{1/2}f)\rangle_{L^2(\R^n)}\\
        &+\langle (-\Delta)^{s/2}m_2,(-\Delta)^{s/2}(\gamma_2^{1/2}f^2 )\rangle_{L^2(\R^n)}\\
        =\,&\langle (-\Delta)^{s/2}(\gamma_1^{1/2}u_f^1-\gamma_2^{1/2}u_f^2),(-\Delta)^{s/2}(\gamma_1^{1/2}f)\rangle_{L^2(\R^n)}\\
        &-\langle (-\Delta)^{s/2}(m_1-m_2),(-\Delta)^{s/2}(\gamma_1^{1/2}f^2 )\rangle_{L^2(\R^n)}.
    \end{split}
\]
Now the conclusion \eqref{eq: relation for solutions} of Lemma \ref{lemma: relation for solutions} (with $W_1=V,W_2=W$) implies that $\gamma_1^{1/2}u_f^1-\gamma_2^{1/2}u_f^2=0$, and hence
\begin{equation}\label{eq: key equation for alternative proof}
    \langle (-\Delta)^{s/2}(m_1-m_2),(-\Delta)^{s/2}(\gamma_1^{1/2}f^2 )\rangle_{L^2(\R^n)}=0
\end{equation}
for any $f \in \widetilde{H}^s(V)$.

Let $f \in \widetilde{H}^s(U)$ where $U \Subset V$ is a nonempty open set, $\phi|_{\bar{U}}=1$, $\phi \in C_c^\infty(V)$, and define $g \vcentcolon= \phi-f \in \widetilde{H}^s(V)$. Now $f^2-g^2=(f+g)(f-g)=\phi(2f-\phi)=2f -\phi^2$. We can use the identity \eqref{eq: key equation for alternative proof} to compute 
\[
\begin{split}
    0&=\langle (-\Delta)^{s/2}(m_1-m_2),(-\Delta)^{s/2}(\gamma_1^{1/2}(f^2-g^2+\phi^2))\rangle_{L^2(\R^n)}\\ 
    &= \langle (-\Delta)^{s/2}(m_1-m_2),(-\Delta)^{s/2}(\gamma_1^{1/2}(2f))\rangle_{L^2(\R^n)}.
\end{split}
\]
This ensures
\begin{equation}\label{eq: UCP formula}
    0=\langle (-\Delta)^{s/2}(m_1-m_2),(-\Delta)^{s/2}(\gamma_1^{1/2}f)\rangle_{L^2(\R^n)}
\end{equation}
for all $f \in \widetilde{H}^s(U)$. By taking $f = \gamma_1^{-1/2}h \in \widetilde{H}^s(U)$, we obtain that
\begin{equation}
    0=\langle (-\Delta)^{s/2}(m_1-m_2),(-\Delta)^{s/2}h\rangle_{L^2(\R^n)}
\end{equation}
for all $h \in \widetilde{H}^s(U)$. Now since $m_1|_{U} = m_2|_{U}$, the UCP implies that $m_1=m_2$ and eventually $\gamma_1=\gamma_2$ in $\R^n$.
\end{proof}

\section{Construction of counterexamples}
\label{sec: counterexamples}

We construct counterexamples to the inverse fractional conductivity equation in this section. These counterexamples are constructed following the strategy introduced in \cite{RZ2022counterexamples}. In bounded domains, these counterexamples are known to exist in the generality presented here (see \cite[Theorem 1.2]{RZ2022counterexamples}). In the case of domains that are bounded in one direction, counterexamples were also constructed when $0 < s < \min(1,n/4]$ for $n \geq 2$ \cite[Theorem 1.3]{RZ2022counterexamples}. The integrability conditions were the main obstruction in the earlier work to construct these counterexamples in general. In the present work, we have changed the integrability requirements from $L^{n/2s}$ to $L^{n/s}$, which allow us to construct counterexamples also in the cases $n/4 < s < 1$ when $n=2,3$, missing in the earlier literature.

We will need Lemma \ref{lemma: ucp of dn map} to verify that the constructed counterexamples have the desired properties. Lemma \ref{lemma: ucp of dn map} is an alternative to \cite[Theorem 1.1 (ii)]{RGZ2022GlobalUniqueness}, which states a similar result with the $H^{2s,\frac{n}{2s}}$ regularity assumption in the place of $H^{s,n/s}$. 

\begin{lemma}[Invariance of data]
\label{lemma: ucp of dn map}
    Let $\Omega\subset \R^n$ be an open set which is bounded in one direction and $0<s<\min(1,n/2)$. Assume that $\gamma_1,\gamma_2\in L^{\infty}(\R^n)$ with background deviations $m_1,m_2$ satisfy $\gamma_1(x),\gamma_2(x)\geq \gamma_0>0$ and $m_1,m_2\in H^{s,n/s}(\R^n) \cap H^s(\R^n)$. Finally, assume that $W_1,W_2 \subset \Omega_e$ are nonempty disjoint open sets and that  $\gamma_1|_{W_1\cup W_2} = \gamma_2|_{W_1\cup W_2}$ holds. Then there holds $\left.\Lambda_{\gamma_1}f\right|_{W_2}=\left.\Lambda_{\gamma_2}f\right|_{W_2}$ for all $f\in C_c^{\infty}(W_1)$ if and only if $m_0 := m_1-m_2 \in H^s(\R^n)$ is the unique solution of
    \begin{equation}
    \label{eq: PDE for m statement}
    \begin{split}
        (-\Delta)^sm+q_{\gamma_2}m&=0\quad\quad\text{in}\quad\Omega,\\
        m&=m_0 \quad\,\text{in}\quad\Omega_e.
    \end{split}
    \end{equation}
\end{lemma}
\begin{proof}
As in the first proof of Theorem~\ref{thm: Global uniqueness}, we may conclude that $\left.\Lambda_{\gamma_1}f\right|_{W_2}=\left.\Lambda_{\gamma_2}f\right|_{W_2}$ for all $f\in C_c^{\infty}(W_1)$ if and only if
$\left.\Lambda_{q_1}f\right|_{W_2}=\left.\Lambda_{q_2}f\right|_{W_2}$ for all $f\in C_c^{\infty}(W_1)$. 

Suppose first that $\left.\Lambda_{\gamma_1}f\right|_{W_2}=\left.\Lambda_{\gamma_2}f\right|_{W_2}$. It follows from the results in \cite{RZ2022unboundedFracCald} that $q_1=q_2$ in $\widetilde{H}^s(\Omega)\times \widetilde{H}^s(\Omega)$. Next note that $q_1=q_2$ in $\widetilde{H}^s(\Omega)\times \widetilde{H}^s(\Omega)$ is equivalent to
    \begin{equation}
    \label{eq: relation m1 and m2}
        \langle (-\Delta)^{s/2}m_1,(-\Delta)^{s/2}(vw/\gamma_1^{1/2})\rangle= \langle (-\Delta)^{s/2}m_2,(-\Delta)^{s/2}(vw/\gamma_2^{1/2})\rangle
    \end{equation}
    for all $v,w\in\widetilde{H}^s(\Omega)$. This implies for $m=m_1-m_2\in H^{s,n/s}(\R^n)\cap H^s(\R^n)$ that
    \begin{equation}\label{eq:computation for m}
    \begin{split}
        &\langle (-\Delta)^{s/2}m,(-\Delta)^{s/2}(vw/\gamma_1^{1/2})\rangle\\
        &=\langle (-\Delta)^{s/2}m_1,(-\Delta)^{s/2}(vw/\gamma_1^{1/2})\rangle-\langle (-\Delta)^{s/2}m_2,(-\Delta)^{s/2}(vw/\gamma_1^{1/2})\rangle\\
        &=\langle (-\Delta)^{s/2}m_2,(-\Delta)^{s/2}(vw(1/\gamma_2^{1/2}-1/\gamma_1^{1/2}))\rangle\\
        &=\langle (-\Delta)^{s/2}m_2,(-\Delta)^{s/2}(vw\frac{\gamma_1^{1/2}-\gamma_2^{1/2}}{\gamma_1^{1/2}\gamma_2^{1/2}})\rangle\\
        &=\langle (-\Delta)^{s/2}m_2,(-\Delta)^{s/2}(vw\frac{m}{\gamma_1^{1/2}\gamma_2^{1/2}})\rangle
    \end{split}
    \end{equation}
    for all $v,w\in\widetilde{H}^s(\Omega)$. Let $w=\gamma_1^{1/2}\phi$, then we obtain
    \[
        \langle (-\Delta)^{s/2}m,(-\Delta)^{s/2}(v\phi)\rangle-\langle (-\Delta)^{s/2}m_2,(-\Delta)^{s/2}(v\phi\frac{m}{\gamma_2^{1/2}})\rangle=0
    \]
    for all $v,\phi\in\widetilde{H}^s(\Omega)$. 
    
    Next let $\phi\in C_c^{\infty}(\Omega)$ and choose $v\in C_c^{\infty}(\Omega)$ such that $v|_{\supp(\phi)}=1$. Then $m$ solves
    \begin{equation}
    \label{eq: PDE for m}
        (-\Delta)^sm+q_{\gamma_2}m=0
    \end{equation}
    in the sense of distributions on $\Omega$, $m \in H^s(\R^n)$ and $m=m_0$ in $\Omega_e$. Since \eqref{eq: PDE for m statement} has a unique solution in $H^s(\R^n)$ by Lemma \ref{eq: well-posedness results and DN maps}, this shows the first direction.
    
    To see the converse, suppose that $m=m_1-m_2$ is the unique solution of \eqref{eq: PDE for m statement}. Now
    \[
        \langle (-\Delta)^{s/2}m,(-\Delta)^{s/2}\phi\rangle-\langle (-\Delta)^{s/2}m_2,(-\Delta)^{s/2}(\phi\frac{m}{\gamma_2^{1/2}})\rangle=0
    \]
for all $\phi \in C_c^\infty(\Omega)$. We may use the computation \eqref{eq:computation for m} with $\phi=vw$, $v,w \in C_c^\infty(\Omega)$ and suitable approximation, to obtain that
\begin{equation}
        \langle (-\Delta)^{s/2}m_1,(-\Delta)^{s/2}(vw/\gamma_1^{1/2})\rangle= \langle (-\Delta)^{s/2}m_2,(-\Delta)^{s/2}(vw/\gamma_2^{1/2})\rangle.
    \end{equation}
This clearly implies that $q_1 = q_2$ in $\widetilde{H}^s(\Omega) \times \widetilde{H}^s(\Omega)$ since $q_1,q_2 \in M(H^s \to H^{-s})$. Therefore, it holds that $u_f^1 = u_f^2$ for any $f \in C_c^\infty(W_1)$ since $(-\Delta)^s+q_1=(-\Delta)^s+q_2$ agree in the weak sense in $\Omega$ and $B_{q_1}(f,\phi)=B_{q_2}(f,\phi)$ for any $\phi \in C_c^\infty(\Omega)$. In particular,
\begin{equation}
    \ip{\Lambda_{q_1}f}{g}=B_{q_1}(u_f^1,g)=B_{q_2}(u_f^2,g)=\ip{\Lambda_{q_2}f}{g}
\end{equation}
for all $f\in C_c^{\infty}(W_1)$, $g \in C_c^\infty(W_2)$ by the support conditions. It follows that
$\left.\Lambda_{q_1}f\right|_{W_2}=\left.\Lambda_{q_2}f\right|_{W_2}$ for all $f \in C_c^\infty(W_1)$, and hence $\left.\Lambda_{\gamma_1}f\right|_{W_2}=\left.\Lambda_{\gamma_2}f\right|_{W_2}$ for all $f\in C_c^{\infty}(W_1)$ as desired.
\end{proof}

We can now prove Theorem \ref{thm: counterexample}. The proof given here follows the one for \cite[Theorem 1.3]{RZ2022counterexamples} with minor changes that are possible due to Lemma \ref{lemma: ucp of dn map}.

\begin{proof}[Proof of Theorem \ref{thm: counterexample}] 
    For any $\delta>0$, we denote by $A_{\delta}$ the open $\delta$-neighbor-hood of the set $A\subset\R^n$ (this should not be confused with the notation $\Omega_e$ for the exterior) and by $(\rho_{\epsilon})_{\epsilon>0}\subset C_c^{\infty}(\R^n)$ the standard mollifiers. First assume that the conductivities $\gamma_1,\gamma_2\in L^{\infty}(\R^n)$ with background deviations $m_1,m_2\in H^{s,n/s}(\R^n)\cap H^s(\R^n)$ satisfy the assumptions of Lemma~\ref{lemma: ucp of dn map} for some $\gamma_0>0$ and $m\vcentcolon =m_1-m_2\in H^s(\R^n)$ solves \eqref{eq: PDE for m statement} with exterior value $m_0\in H^s(\R^n)$. Then by the computation in \cite[Proof of Theorem 1.3]{RZ2022counterexamples} we know that $m_1\in H^s(\R^n)$ solves 
    \[
    (-\Delta)^sm_1-\frac{(-\Delta)^sm_2}{\gamma_2^{1/2}}m_1=\frac{(-\Delta)^sm_2}{\gamma_2^{1/2}}\quad \text{in}\quad \Omega,\quad m_1=m_2+m_0\quad\text{in}\quad \Omega_e.
    \]
    Now let $\gamma_2\equiv 1$. Then $m_1\in H^s(\R^n)$ is a $s$-harmonic function in $\Omega$; more precisely $m_1$ solves
    \begin{equation}
    \label{eq: constr m}
        (-\Delta)^sm_1=0\quad \text{in}\quad \Omega,\quad m_1=m_0\quad\text{in}\quad \Omega_e.
    \end{equation}
    
    Now choose $\omega\Subset \Omega_e\setminus\overline{W_1\cup W_2}$ and let $\epsilon>0$ be such that the sets $\Omega_{5\epsilon},\omega_{5\epsilon}\subset \R^n\setminus (W_1\cup W_2)$ are disjoint. Let $\eta\in C_c^{\infty}(\omega_{3\epsilon})$ be a nonnegative cutoff function satisfying $\eta|_{\overline{\omega}_{2\epsilon}}=1$. By the Lax--Milgram theorem and the fractional Poincar\'e inquality on domains bounded in one direction (cf.~\cite[Theorem 2.2]{RZ2022unboundedFracCald}) there is a unique solution $\widetilde{m}_1\in H^s(\R^n)$ to
    \begin{equation}
    \label{eq: PDE in extended domain}
        (-\Delta)^s\widetilde{m}_1=0\quad \text{in}\quad \Omega_{2\epsilon},\quad \widetilde{m}_1=\eta\quad\text{in}\quad \R^n\setminus\overline{\Omega}_{2\epsilon}.
    \end{equation}
    Proceeding as in \cite[Proof of Theorem 1.2 and 1.3]{RZ2022unboundedFracCald} one can show that
    \[
        m_1\vcentcolon =C_{\epsilon}\rho_{\epsilon}\ast\widetilde{m}_1\in H^s(\R^n)\quad\text{with}\quad C_{\epsilon}\vcentcolon=\frac{\epsilon^{n/2}}{2|B_1|^{1/2}\|\rho\|_{L^{\infty}(\R^n)}^{1/2}\|\widetilde{m}_1\|_{L^2(\R^n)}}
    \]
    solves
    \[
        (-\Delta)^sm=0\quad \text{in}\quad \Omega,\quad m=m_1\quad\text{in}\quad \Omega_e.
    \]
    Moreover, $m_1$ has the following properties
\begin{enumerate}[(i)]
    \item\label{item 2 m} $m_1\in L^{\infty}(\R^n)$ with $\|m_1\|_{L^{\infty}(\R^n)}\leq 1/2$,
    \item\label{item 3 m} $m_1\in H^{s}(\R^n)\cap H^{s,n/s}(\R^n)$
    \item\label{item 4 m} and $\supp(m_1)\subset \R^n\setminus(W_1\cup W_2)$.
\end{enumerate}
The statement \ref{item 3 m} follows from Young's inequality since $n/s>2$ implies $p\vcentcolon =\frac{2n}{n+2s}\in (1,\frac{n}{2s})$ and $1/p+1/2=1+s/n$. Moreover, the support conditions imply that $\gamma_1=1$ in $W_1\cup W_2$.

This ensures that the conductivity $\gamma_1$ defined by $\gamma_1^{1/2}\vcentcolon = m_1+1$ and the background deviation $m_1$ satisfy all required properties but $\gamma_1\neq \gamma_2$. Now since $W_1,W_2\subset \Omega_e$ are two disjoint open sets, we have found two conductivities $\gamma_1,\gamma_2$ satisfying the properties of Lemma~\ref{lemma: ucp of dn map} and $m\vcentcolon = m_1-m_2$ is the unique solution to \eqref{eq: PDE for m statement}, which in turn implies that the induced DN maps fulfill $\Lambda_{\gamma_1}f|_{W_2}=\Lambda_{\gamma_2}f|_{W_2}$ for all $f\in C_c^{\infty}(W_1)$.
\end{proof}

\bibliography{Refs} 

\bibliographystyle{alpha}

\end{document}